\documentclass[12pt,twoside]{amsart}
\usepackage{amsfonts}
\usepackage{fancyhdr}
\usepackage{amsmath} 
\usepackage{hyperref}
\usepackage{lscape}
\usepackage{color}

\usepackage{amssymb}
\theoremstyle{plain}

\makeatletter
\@addtoreset{equation}{section}

\makeatother
\newtheorem{theo}{Theorem}[section]
\newtheorem{pro}[theo]{Proposition}
\newtheorem{lemm}[theo]{Lemma}
\newtheorem{cor}[theo]{Corollary}

\newtheorem{rem}[theo]{Remark}
\newtheorem{deff}[theo]{Definition}

\newcommand{\avert}[1]{-\hskip-0.41cm\int_{#1}}

\numberwithin{equation}{section}

\begin{document} 
\allowdisplaybreaks

\title[$L^p$ estimates for magnetic Schr\"odinger operators]{Maximal inequalities and Riesz transform estimates on $L^p$ spaces for magnetic Schr\"odinger operators II}
\author{Besma Ben Ali}
\address{ B. Ben Ali
\\
Universit\'e de Paris-Sud, UMR du  CNRS  8628
\\
91405 Orsay Cedex, France} \email{besma.ben-ali@math.u-psud.fr}

\begin{abstract}
The paper concerns the  magnetic Schr\"odinger operator $H(\textbf{a},V)=\sum_{j=1}^{n}(\frac{1}{i}\frac{\partial}{\partial
x_{j}}-a_{j})^{2}+V $ on $\mathbb{R}^n$. We prove some $L^p$ estimates on the Riesz transforms of $H$ and we establish some related maximal inequalities.
The conditions that we arrive at, are essentially based on the control of the magnetic field by the electric potential.

\end{abstract}
\date{}
\maketitle

\begin{quote}
{\tableofcontents}
\end{quote}
\vline

\section{Introduction}
Consider the Schr\"odinger operator with magnetic field
\begin{equation}\label{def}H(\textbf{a},V)=\sum_{j=1}^{n}(\frac{1}{i}\frac{\partial}{\partial
x_{j}}-a_{j})^{2}+V \, \textrm{in}\, \, \mathbb{R}^{n},\qquad
n\geq 2,
 \end{equation}
  where $\textbf{a}=(a_{1},a_{2},\cdots,a_{n}):
\mathbb{R}^{n}\rightarrow \mathbb{R}^{n}$ is the magnetic potential and $V:\mathbb{R}^{n}\rightarrow \mathbb{R}$ is the electric potential.
Let \begin{equation}B(x) = curl\, \textbf{a}(x)= (b_{jk}(x))_{1\leq j,k\leq n}\end{equation}
 be the magnetic field generated by $\textbf{a}$, where
 \begin{equation}
 b_{jk} = \frac{\partial a_{j}}{\partial x_{k}}- \frac{\partial a_{k}}{\partial x_{j}}.
 \end{equation} 
 We will assume that $\textbf{a}\in L^{2}_{loc}(\mathbb{R}^n)^n$ and $V\in L^{1}_{loc}(\mathbb{R}^n), \,\, V\geq 0$.
Let
 \begin{equation} L_{j}=
\frac{1}{i}\frac{\partial}{\partial x_{j}}-a_{j} \qquad
\textrm{for}\qquad 1\leq j\leq n, 
 \end{equation} 
 Set $L=(L_{1},\ldots, L_{n})$ and $
|L u(x)|=\big{(} \sum_{j=1}^{n} |L_{j}u(x)|^{2}\big{)}^{1/2}.$
 
Note that $ L^{\star}_{j}=L_{j}$ for all $1\leq j\leq n,$. Let
$$L^{\star}=(L^{\star}_{1},\ldots, L^{\star}_{n})^{T}.$$
We define the form $\mathcal{Q}$ by
\begin{equation}\mathcal{Q}(u,v)= \sum_{k=1}^{n} \int_{\mathbb{R}^n}
L_{k} u . \overline{L_{k} v} dx+\int_{\mathbb{R}^n}V  u . \bar{ v} dx ,
\end{equation} 
with domain $\mathcal{D}(\mathcal{Q})=\mathcal{V}\times \mathcal{V}$ where $$\mathcal{V} = \{u\in
L^{2}, L_{k}u \in L^{2} \,\, \textrm{for } k=1,\ldots,n \,\,\textrm{and}\, \,  \sqrt{V} u \in L^{2}\}.$$
 Let $\dot{\mathcal{V}}$ be the closure of 
$C^{\infty}_{0}\mathbb{R}^n$ under the semi-norm $$
\|f\|_{\dot{\mathcal{V}}}= \big( \|L f\|_{2}^2 + \|V^{1/2} f\|_{2}^2\big)^{1/2}.
$$
 We denote $H(\textbf{a},V)=H$, the self-adjoint operator on $L^2(\mathbb{R}^n)$ associated to this symmetric and closed form.
 
  The domain of $H $ is given by:
 $$\mathcal{D}(H)=\{u \in \mathcal{D}(\mathcal{Q}), \exists v\in L^{2} \,\, \textrm{so that}\,\, \mathcal{Q}(u,\phi)=\int_{\mathbb{R}^n} v\bar{\phi} dx, \, \forall \phi \in \mathcal{D}(\mathcal {Q}) \}.$$

 The operators $L_{j} H(\textbf{a},V)^{-1/2}$ are called the Riesz transforms associated with $H(\textbf{a},V)$.
 We know that
 \begin{equation}\label{LL2} 
 \sum^{n}_{j=1}\| L_{j} u\|^{2}_{2}+ \|V^{1/2}u\|^{2}_{2} = \| H(\textbf{a},V)^{1/2} u\|^{2}_{2}, \qquad \forall u \in \mathcal{D}(\mathcal{Q}) =\mathcal{D}(H(\textbf{a},V)^{1/2}).
 \end{equation}
 Hence, the operators $L_{j} H(\textbf{a},V)^{-1/2}$ are bounded on $L^{2}(\mathbb{R}^{n})$, for all $j=1,\ldots,n$.

 The aim of this paper is to establish the $L^p$  boundedness of the operators $L_{j} H(\textbf{a},V)^{-1/2}$ and $V^{\frac{1}{2}}H(\textbf{a},V)^{-\frac{1}{2}}$. 

In the case where the magnetic potential is absent, that is, $H(\textbf{a},V)=-\Delta+V$ and $LH(\textbf{a},V)^{-\frac{1}{2}}=\nabla(-\Delta+V)^{-\frac{1}{2}}$, many important studies have been established. We mention the works of Helffer-Nourrigat \cite{HNW}, Guibourg\cite{Gui2} and Zhong \cite{Z}, in which they considered the case of polynomial  potentials. A generalization of their results was given by Shen \cite{Sh1}, he proved the $L^p$ boundedness of Riesz transforms of Schr\"odinger operators with electric potential contained in certain reverse H\"older classes. Auscher and I improved this result in \cite{AB}, using a different approach based on local estimates. Note that this approach can be extended to more general spaces for instance some Riemannian manifolds and Lie groups( see \cite{BB}).

In the presence of the magnetic field, we know that these operators are of weak type (1.1) and hence, by interpolation, are  $L ^p$ bounded for all $ 1 <p \leq  2$.
This result was proved by Sikora using  the finite speed propagation property \cite{Sik}. Independantly, Duong, Ouhabaz and Yan \cite{DOY} have proved the same result using another method.\\

The main purpose of this work is to find sufficient conditions on the electric potential and the magnetic field, for which the Riesz transforms of $H(\textbf{a},V)$ are $L^p$ bounded for the range $p>2$. 
 Note that, because of the gauge invariance of the operator $H(\textbf{a},V)$ and the nature of the $L^{p}$ estimates, any such quantitative condition should be imposed on the magnetic field $B$ , not directly on $\textbf{a}$. 

 In a previous paper \cite{Be}, many important results about this problem were established. 
 Under certain conditions used by Shen in \cite{Sh4} and given in terms of the reverse H\"older inequality on the magnetic field and the electric potential, we proved that the Riesz transforms of the pure magnetic Schr\"odinger operator $H(\textbf{a},0)$ are $L^p$ bounded for all $p\geq 2$. We have also extended the results of \cite{AB} about $-\Delta+V$ to the magnetic operator $H(\textbf{a},V)$. 

 The second aim of this article is to establish important maximal inequalities related to the $L^p$ behaviour of $L_{j}L_{k} H(\textbf{a},V)^{-1}$, $V^{1/2}L H(\textbf{a},V)^{-1}$ and other operators called the second order Riesz transforms. Estimates on these operators are of great interest in the study of spectral theory of $H(\textbf{a},V)$. 
 There are rather few works around the behaviour of these operators. We cite Guibourg who considered the polynomial case and established an $L^2$ estimate \cite{Gui1}. Shen \cite{Sh4}, generalised \cite{Gui1} and proved under reverse H\"older conditions, the $L^p$ boundedness of $L_{j}L_{k} H(\textbf{a},V)^{-1}$.  Independantly and under the same conditions, we have proved and generalised the results of Shen in \cite{Be}.
Note that in \cite{Sh4} and \cite{Be}, the contribution of the magnetic field was controlled by introducing an auxiliary function $m(.,\omega)$ defined by Shen \cite{Sh1} for $RH_{\infty}$ class, he generalizes an early version of a useful auxiliary function for polynomial potentials.\\
In this paper we will use another approach, the contribution of the magnetic field will be controlled by the electric potential and the magnetic Schr\"odinger operater will be treated us a perturbation of the Laplace operator $-\Delta$.

Here, our assumptions on potentials will be given in terms of reverse H\"older inequality. Let us recall the definition of these weight classes:
  
 \begin{deff}
 Let $\omega \in L^q_{loc}(\mathbb{R}^n)$,  $\omega>0$ almost everywhere, $\omega\in RH_{q}$, $1<q\le\infty$, the class of the reverse H\"older weights with exponent $q$, if there exists a constant $C$ such that  for any cube $Q$ of $\mathbb{R}^n$,

 \begin{equation}\label{2RHclass}
 \Big{(} \avert{Q} {\omega^{q}}(x)dx \Big{)}^{1/q}\leq C\Big{(}\avert{Q} \omega (x)dx \Big{)}.
 \end{equation}
If $q=\infty$, then the left hand side is the essential supremum on $Q$. 
The smallest $C$ is called the $RH_{q}$ constant of $\omega$.
 \end{deff}
\textit{A note about notations:}
Throughout this paper we will use the following notation $\avert {Q}\omega=\frac{1}{|Q|}\int_{Q}\omega.$
$C$ and $c$ denote constants. As usual, $\lambda Q$ is the cube co-centered with $Q$ with sidelength $\lambda$ times that of $Q$.

We now state our main result:
  \begin{theo}\label{th:1a'}Suppose $\textbf{a}\in L^{2}_{loc}(\mathbb{R}^{n})^{n}$ and $V  \in RH_{q}$, $1<q\le +\infty$. Also assume that there exists a constant $C>0$ such that for any cube $Q$ in $\mathbb{R}^{n}$:

 \begin{equation}\label{eq:star}
 \left\{\begin{array}{ll}
 \sup_{Q} |B| \leq C \avert {Q} V, \\ \sup_{Q} | \nabla B |\leq C  (\avert{Q} V)^{3/2},
  \end{array}
 \right.
 \end{equation}  
where $| B |=\sum_{j,k} | b_{jk} |$ and $\nabla=(\frac{\partial}{\partial x_{1}},\ldots,\frac{\partial}{\partial x_{n}})$ . Then, there exists an $\epsilon >0$ such that for any $1\leq p<q^{\star}+\epsilon$, there exists a constant $C_{p}>0$, depending on $V$ such that
 \begin{equation}
 \label{eq:Lp}  \| LH(\textbf{a},V)^{-\frac{1}{2}}(f)\|_{p}\leq C_{p} \| f\|_{p},
 \end{equation}
 for any $f\in C_{0}^{\infty}(\mathbb{R}^{n})$ if $p>1$,\\ and
 $$|\{x\in \mathbb{R}^n\, ; \,  |L f(x)| > \alpha\}| \le \frac{C_{1}}{\alpha} \|H(\textbf{a},V)^{\frac{1}{2}}f\|_{1},$$
 for any $\alpha>0$ and $f\in C_{0}^{\infty}(\mathbb{R}^{n})$ if $p=1$.

  \end{theo}
 \begin{rem}\begin{enumerate}
\item Note that condition \eqref{eq:star} implies the following inequalities:
 \begin{equation}\label{eq:star'}
 \left\{\begin{array}{ll}
  |B| \leq C  V, \\  | \nabla B | \leq C   V^{3/2},
  \end{array}
 \right.
 \end{equation} 
 almost everywhere in $ \mathbb{R}^n$. These  hypotheses are not sufficient to obtain \eqref{eq:Lp}.
\item Condition \eqref{eq:star} includes the polynomial case.
\end{enumerate}
 \end{rem}
  
 An important step to prove the previous result is to establish the following reverse estimates that hold an importance of their own :
\begin{theo}\label{th:1b'}Suppose $\textbf{a}\in L^{2}_{loc}(\mathbb{R}^{n})^{n}$, $V\in A_{\infty}$. Also assume that there exists a constant $C>0$ such that for any cube  $Q$ in $\mathbb{R}^n$ :
\begin{equation}\label{eq:starstar}  \sup_{Q} |B| \leq C \avert {Q} V.
\end{equation}
Then, for any $1\leq p<\infty$, there exists $C_{p}>0$,
which depends on \eqref{eq:starstar} and $V$,
such that
\begin{equation}\label{eq:revLp}
\|H(\textbf{a},V)^{\frac{1}{2}}(f)\|_{p}\leq C_{p}\big{(}\| Lf\|_{p}+ \| V^{\frac{1}{2}}\,f\|_{p}\big{)} ,
\end{equation} 
 for any $f\in C_{0}^{\infty}(\mathbb{R}^{n})$, if $p>1$,\\ and
\begin{equation}\label{eq:wt}|\{x\in \mathbb{R}^n\, ; \,  |H(\textbf{a},V)^{\frac{1}{2}} f(x)| > \alpha\}| \le \frac{C_{1}}{\alpha} \int |L f| + V^{\frac{1}{2}}|f|,
\end{equation}
for any $\alpha>0$ and $f\in C_{0}^{\infty}(\mathbb{R}^{n}),$ if $p=1$.
 \end{theo} 
 
 Along with the study of Riesz transforms, we will also establish some maximal inequalities:  
  \begin{theo}\label{max}Suppose $\textbf{a}\in L^{2}_{loc}(\mathbb{R}^{n})^{n}$, $V\in L_{loc}^{1}(\mathbb{R}^n)$ and $V \in RH_{q}$, $1<q\le +\infty$. Also assume that there exists a constant $C>0$ such that for any cube $Q\subset \mathbb{R}^{n}$
 \begin{equation}
 \left\{\begin{array}{ll}
 \sup_{Q}| B |\leq C\avert {Q}V \\ \sup_{Q} |  \nabla B |\leq C(\avert {Q}V)^{3/2}.
  \end{array}
  \right.
 \end{equation}
 Then, there exists an $\epsilon>0$ depending on $V$, such that for every $s=1,\ldots.,n$ and $ k=1,\ldots,n$, and for any $1p<q+\epsilon$, there exists a constant $C_{p}>0$ such that for any $f\in C_{0}^{\infty}(\mathbb{R}^{n}),$
 \begin{equation}\| L_{s}L_{k}(f)\|_{q}\leq C_{q} \|H(\textbf{a},V) f\|_{q}.
 \end{equation}
 
\end{theo}
The proof of this theorem will use Theorem \ref{th:1a'} and \ref{th:1b'}. We also need to study the $L^p$ boundedness of different second order Riesz Transforms us $VH(\textbf{a},V)^{-1}$, $H(\textbf{a},0)H(\textbf{a},V)^{-1}$ (studied in section 5) and especially the behaviour of operator $V^{\frac{1}{2}}LH(\textbf{a},V)^{-1}$ described in the following theorem: 

\begin{theo}\label{vl}
Suppose $\textbf{a}\in L^{2}_{loc}(\mathbb{R}^{n})^{n}$ and $V  \in RH_{n/2}$. Also assume that there exists a constant $C>0$ such that for any cube $Q$ in $\mathbb{R}^{n}$:

 \begin{equation}
 \left\{\begin{array}{ll}
 \sup_{Q} |B| \leq C \avert {Q} V, \\ \sup_{Q} | \nabla B |\leq C  (\avert{Q} V)^{3/2}.
  \end{array}
 \right.
 \end{equation}.
Then the operator $V^{\frac{1}{2}} L H(\textbf{a},V)^{-1}$ is $L^p$ bounded for
$$
\begin{cases}
1\le p < \frac{2(q+\epsilon)n}{3n-2(q+\epsilon)} , \, \,  {\rm if} \  q<n
\\
1\le p < 2(q+\epsilon), \, \,  {\rm  if} \  q\ge n
\end{cases}
$$
where $\epsilon$ depends only on $V$.
\end{theo}
Note that this result was proved by Shen when the magnetic potential is absent (see Theorem 4.13, \cite{Sh1} that can be recovered by \cite{AB} methods under the same hypotheses and for $n\geq 1$ instead of $n\geq 3$).\\

 We mention without proof that our results admit local versions, replacing $V\in RH_{q}$ by  $V\in RH_{q,loc}$ which is defined by the same conditions on cubes with sides less than 1. Then we get the corresponding results and estimates  for $H+1$ instead of $H$.  The results on operator domains are valid under local assumptions.

The arguments are based on local estimates. We briefly sketch the main tools :

1) An improved Fefferman-Phong inequality for $A_{\infty}$ potentials.

2) Criteria for proving  $L^p$ boundedness of operators in absence of kernels. 

3)  Mean value inequalities for  nonnegative subharmonic functions against $A_{\infty}$ weights.

4) Complex  interpolation, together with  $L^p$ boundedness of imaginary powers 
of $H(\textbf{a},V)$ for $1<p<\infty$. 

5) A Calder\'on-Zygmund decomposition adapted to level sets of the maximal function of $|L f| + |V^{1/2}f| $. 

6) A gauge transform adapted to the reverse H\"older conditions on the potentials.

7) Reverse H\"older inequalities involving $\L u$, $|B|^{1/2} u$ and $V^{1/2} u$ for weak solutions
of $H(\textbf{a},V) u=0$.

 The paper is organized as follows. In Section 2 we give the principal tools to prove the theorems mentioned above. We state an improved Fefferman-Phong inequality and we establish an adapted gauge transform. Section 3 is devoted to  establish some reverse estimates via a Calder\'on-Zygmund decomposition. In section 4 we give different estimates for the weak solution of $H(\textbf{a},V) u=0$. We state some useful maximal inequalities in section 5. Section 6 is concerned with the proof of Theorem \ref{th:1a'}. Finally, in section 7, we study the operator $V^{\frac{1}{2}}LH^{-1}$ and give the proof of Theorem \ref{max}.

\section{Preliminaries}

We begin by recalling some properties of the reverse H\"older classes.
 \begin{pro}\label{11.1}(Proposition 11.1 [AB]) Let $\omega$ be a nonnegative measurable function. Then the following are equivalent:
\begin{enumerate}
  \item $\omega\in A_{\infty}$.
  \item For all $s\in (0,1)$, $\omega^s\in B_{1/s}$.
  \item There exists $s\in (0,1)$, $\omega^s\in B_{1/s}$.
\end{enumerate}

\end{pro}
\begin{rem}It is well known that if $\omega \in RH_q$ and $q<+\infty$, then $\omega \in RH_{p}$ for all $1<p<q$  and there exists an $\varepsilon >0$ such that $\omega\in RH_{q+\varepsilon}$.\\
 We also know that $\omega \in A_{\infty}$ if and only if there exists $q>1$ such that $\omega\in RH_{q}$.\\ Here $A_{\infty}$ is the Muckenhoupt weight class, defined as the union of all $A_p$, $1\leq p<\infty.$
 If $\omega\in A_{\infty}$ then $\omega(x) dx$ is a doubling measure (see  \cite{St},chap V for more details).
\end{rem}

The first step of this work is to use the properties of the $A_{\infty}$ weights to establish some
reverse H\"older inequalities with the weak solutions. Then, we apply the following  criterion for $L^p$ boundedness \cite{AM1}( A slightly weaker version appears in Shen \cite{Sh2}).
 \begin{theo} \label{theor:shen}
 
 Let $1\le p_{0} <q_0\le \infty$. Suppose that $T$ is a bounded
sublinear operator  on $L^{p_{0}}(\mathbb{R}^n)$. Assume that there exist
constants $\alpha_{2}>\alpha_{1}>1$, $C>0$ such that
\begin{equation}\label{T:shen}
\big(\avert {Q} |Tf|^{q_0}\big)^{\frac1{q_0}}
\le
C\, \bigg\{ \big(\avert {\alpha_{1}\, Q}
|Tf|^{p_0}\big)^{\frac1{p_0}} +
(S|f|)(x)\bigg\},
\end{equation}
for all cube $Q$, $x\in Q$ and  all $f\in L^{\infty}_{comp}(\mathbb{R}^n)$ with 
support in $\mathbb{R}^n\setminus \alpha_{2}\, Q$, where $S$ is a positive operator.
Let $p_{0}<p<q_{0}$. If $S$ is bounded on $L^p(\mathbb{R}^n)$, then, there is a constant $C$ such that
$$
\|T f\|_{p}
\le
C\, \|f\|_{p}
$$
for all    $f\in L_{comp}^{\infty}(\mathbb{R}^n)$.\end{theo}

Fix an open set $\Omega$ and $f\in L^{\infty}_{comp}(\mathbb{R}^{n}) $, the space of compactly supported bounded functions on $\mathbb{R}^n$. By a weak solution of 
  \begin{equation}\label{12}
H(\textbf{a},V) u =f\,\, \textrm{in}\,\Omega,
\end{equation}
we mean $u\in W(\Omega)$, with $$W(\Omega)=\{u\in L^1_{loc}(\Omega)\, ;\,     V^{1/2}u \, and \, L_{k} u \in L^2_{loc}(\Omega) \, \forall k=1,\ldots,n\}$$ and the equation \eqref{12} holds in the sense of distribution on $\Omega$. We note that if $u\in W(\Omega) $, then by Poincar\'e and the diamagnetic inequalities, $u \in L^2_{loc}(\Omega)$. 

The weak solution satisfies some important inequalities which will be useful to prove our results:

 \begin{lemm}\textbf{Caccioppoli type inequality}\\
Let $u$ a weak solution of $H(\textbf{a},V)u=f$ in $2Q$, where $Q$ is a cube of $\mathbb{R}^n$ and $f\in L^{\infty}_{comp}(\mathbb{R}^{n})$. Then
 \begin{equation}\label{eq:cc} \int_{Q} |Lu|^{2}+V|u|^{2}  \leq C\{ \int_{2Q} |f||u|  + \frac{1}{R^{2}} \int_{2Q}  |u|^{2} \}.
 \end{equation}

We also give some important tools:
 \end{lemm}
 \begin{pro}\textbf{Diamagnetic inequality}\cite{LL}\\
For all $u\in W^{1,2}_{\textbf{a}}(\mathbb{R}^{n})$, with $$W^{1,2}_{\textbf{a}}(\mathbb{R}^{n})=\{u\in L^{2}(\mathbb{R}^n),\,\, L_{k} u \in L^{2}(\mathbb{R}^{n}),\,\, k=1\cdots , n\},$$
we have
\begin{equation}\label{2diam}|\nabla(|u|)|\leq| L(u)|.
\end{equation}

\end{pro}
\begin{pro}\textbf{Kato-Simon inequality}:
 \begin{equation}\label{eq:KS}| (H(\textbf{a},V)+\lambda)^{-1} f| \leq (-\Delta+ \lambda)^{-1}|f|; \qquad \forall f\in L^{2}(\mathbb{R}^n),\,\forall \lambda >0.
 \end{equation}
We also have the following domination inequality \cite{Si4}:
 \begin{equation}\label{dom} | e^{-tH(\textbf{a},V)} f| \leq e^{\Delta t}|f|; \qquad \forall t\geq 0, \, f\in L^{2}(\mathbb{R}^n).
 \end{equation}
\end{pro}
 \begin{center}{\textbf{Fefferman-Phong inequalities}}
 \end{center}

The usual Fefferman-Phong inequalities are of the form:
 \begin{equation}\label{eq:feff}
  \int_{Q} | u |^{p}  \min\{\avert {Q}\omega, \frac{1}{R^p}\}\leq C\{\int_{Q} | Lu|^{p}  + \omega | u |^{p} \}.
 \end{equation}

In \cite{AB} we established an improved version for these inequalities in absence of the magnetic potential. We can extend this improvement to the magnetic Schr\"odinger operators:
\begin{lemm}\label{FP}\textbf{An improved Fefferman-Phong inequality} :\\ Let $\omega \in A_\infty$ and $1\le p<\infty$. Then there are constants $C>0$ and $\beta\in (0,1)$ depending only
on $p$, $n$ and the $A_\infty$ constant of $w$ such that for all cubes $Q$ $($with sidelength $R)$ and $u \in C^1(\mathbb{R}^{n})$, one has
\begin{equation}\label{eq:FP}
\int_Q |L u|^p + \omega|u|^p \ge \frac{C m_{\beta}({R^p  \avert{Q} \omega})}{ R^{p}} \int_Q |u|^p
 \end{equation}
where $m_{\beta}(x)  = x$ for $x\le 1$ and $m_{\beta}(x) = x^\beta$ for $x\ge 1$. 
 
\end{lemm}
	
The proof is the same as that of Lemma 2.1 in \cite{AB}, combined with the diamagnetic inequality.
\begin{center}{\textbf{Iwatsuka Gauge transform}}

\end{center}

\begin{lemm}\label{th:jauge}

 Let $\textbf{a}\in L^{2}_{loc}(\mathbb{R}^{n})^{n}$ and $Q$ a cube of $\mathbb{R}^n$. Suppose $B\in C^{1}(\mathbb{R}^{n}, M_{n}(\mathbb{R}))$. Then there exist $\textbf{h}\in C^{1}(Q,\mathbb{R}^{n})$ and a real function $\phi \in C^{2}(Q)$, 
 such that $curl \textbf{h}=B$ in $Q$ and
\begin{equation}\label{Gau}\textbf{h}=\textbf{a}- \nabla \phi,\qquad \textrm{in} \,\, Q,
\end{equation}
with
\begin{equation}\label{jaugebes1}\sup_{ Q} | \textbf{h}|\leq  C\,R\,\sup_{Q} |B|,
\end{equation}
and
\begin{equation}\label{jaugebes2}\sup_{ Q} | \nabla\textbf{h}|\leq  C\big(\sup_{Q} |B|+R\,\sup_{Q} |\nabla B|\big).
\end{equation}
\end{lemm}
See the proof of Lemma 2.4 in \cite{Sh5} which uses the construction of Iwatsuka \cite{I}.

\begin{lemm}\label{th:jauge'}

 Let $Q$ a cube in $\mathbb{R}^n$. Suppose $B\in L^{\infty}(Q, M_{n}(\mathbb{R}))$. Then there exist $\textbf{h}\in L^{\infty}(Q,\mathbb{R}^{n})$ and a real function $\phi \in W^{1,\infty}(Q)$, 
 such $curl \textbf{h}=B$ and
\begin{equation}\label{phi'}\textbf{h}=\textbf{a}- \nabla  \phi \qquad\textrm{a.e in}\, Q ,
\end{equation}
and
 \begin{equation}\label{jaugebes'}\sup_{ Q} | \textbf{h}|\leq  C\,R\,\sup_{Q} |B|,
\end{equation}
and
\begin{equation}\label{jaugebes''}\sup_{ Q} | \nabla\textbf{h}|\leq  C\big(\sup_{Q} |B|+R\,\sup_{Q} |\nabla B|\big).
\end{equation}
\end{lemm}
\begin{proof}

Let $(\textbf{a}_{m})_{m\geq 0}$ be the sequence of $C^{1}$ functions obtained by convolution with $\textbf{a}$ and converge in $L^{2}_{loc}$ to $\textbf{a}$. Set $(B_{m})_{m\geq 0}$, $(\phi_{m})_{m\geq 0}$ and $(\textbf{h}_{m})_{m\geq 0}$ as the corresponding sequences of the Lemma \ref{th:jauge}. Note that $(\textbf{h}_{m})_{m\geq 0}$ converges in $L^{n}(Q,\mathbb{R}^n)$. Let $\textbf{h}$ be this limit, it satisfies \eqref{phi'}.
Note also that $(B_{m})_{m\geq 0}$ converges to $B$ in $L^{n/2}_{loc}(Q, M_{n}(\mathbb{R}))$ and  $curl \textbf{h}= B$ holds almost everywhere in $Q$, where curl is defined in the sens of distribution. 
 
We know that for all $ m\geq 0$, $$ (\avert{Q} | \textbf{h}_{m}|^{n})^{1/n}\leq  c\,R(\avert{Q} |B_{m}|^{\frac{n}{2}})^{\frac{2}{n}},$$
uniformly in $m$. Then 
applying the limit, we obtain
 $$(\avert{Q} | \textbf{h}|^{n})^{1/n}\leq  c\,R(\avert{Q} |B|^{\frac{n}{2}})^{\frac{2}{n}}.$$
 Hence inequality \eqref{jaugebes'} follows easily. By a similar argument, inequality \eqref{jaugebes''} holds.
 \end{proof}

 \section{Reverse estimates}
The present section talks about certain tools that are handy in the proof of Theorem \ref{th:1b'}.
 Note that this theorem can be obtained as a consequence of Theorem 1.6 in \cite{Be} if we also assume that  $|B|$ is in $RH_{n/2}$. However, condition \eqref{eq:starstar} is sufficient to obtain estimate \eqref{eq:revLp}.

By duality, the $L^p$ boundedness of Riesz transforms for $1<p\leq 2$ (proved by \cite{Sik} and \cite{DOY}) implies the estimate \eqref{eq:revLp} for any $p\geq 2$. For $p< 2$, 
we follow step by step the proof of the Theorem 1.2 of \cite{AB} once the appropriate Calder\'on-Zygmund decomposition \ref{CZ'} is established. We also use the fact that the time derivatives  of the kernel of semigroup $ e ^{-tH} $ satisfy Gaussian estimates (see \cite{CD}, \cite{Da}, \cite{G} and \cite{Ou} or, theorem 6.17).

  Let us introduce the main technical lemma of this work, which in itself is an interesting result:

 \begin{lemm}\label{CZ'} Let $1\leq p<2$ and $\alpha>0$. Under the assumptions of Theorem \ref{th:1b'}, we have:
 for any function $f\in C^{\infty}_{0}(\mathbb{R} ^n)$ such that $$ \| L f \|_{p}+\| V^{\frac{1}{2}}f \|_{p} <\infty.$$
 Then, one can find a collection of cubes $(Q_k)$ and functions $g$ and $b_k$  such that  
 \begin{equation} \label{eq:cza'}f=g+\sum_{k} b_k
 \end{equation}
and the following properties hold:
 \begin{equation}\label{eq:czb'}\| L g \|_{2}+\| V^{\frac{1}{2}} g \|_{2} \leq C\alpha^{1-\frac{p}{2} }( \|  Lf \|_{p}+\| V^{\frac{1}{2}} f \|_{p} )^{p/2}
 \end{equation}
 \begin{equation}\label{eq:czc'} \int_{Q_k}| L b_{k} |^{p}+ R_{k}^{-p} | b_{k}|^{p}  \leq C\alpha^{p} | Q_{k} |
 \end{equation}
 \begin{equation} \label{eq:czd'}\sum_{k}| Q_k | \leq C\alpha ^{-p}(\int_{\mathbb{R} ^{n}} | L f |^{p}+| V^{\frac{1}{2}}f |^{p} )
 \end{equation}
 \begin{equation}\label{eq:cze'}\sum_{k} \mathbf{1}_{Q_k} \leq N,
 \end{equation}
where 
 $N$ depends only on  the dimension and $C$ on the dimension,  $p$ and the $RH_{n/2}$ constant of $|B|$. Here, $R_{k}$ denotes the sidelength of $Q_{k}$ and gradients are taken in the sense of distributions in $\mathbb{R}^n$.

 \end{lemm}
 \begin{rem}We establish an improved version of estimate \eqref{eq:czb'}:
 \begin{equation}\label{Lg'}
 \|Lg\|_{\infty}\leq C\alpha.
 \end{equation}
 \end{rem}
 \begin{proof} Let $\Omega$ be the open set $\{ x\in \mathbb{R}^{n}; M\big{(}| L f|^{p} +| V^{\frac{1}{2}}f |^{p} \big{)}(x)>\alpha^{p} \}$, where $M$ is the uncentered maximal operator over the cubes of $ \mathbb{R}^{n}$. If $\Omega$ is empty, then set $g=f$ and $b_{i}=0$. Otherwise, our argument is subdivided into six steps.\\
\\
 \textbf{a) Construction of the cubes:}\\
 \\
The maximal theorem gives us $$| \Omega | \leq C{\alpha}^{-p}\int_{\mathbb{R}^n} | L f|^{p} +| V^{\frac{1}{2}}f |^{p}<\infty .$$
 Set $F=\mathbb{R}^n\setminus \Omega$. Using the Lebesgue derivation Theorem, we have
 \begin{equation}| Lf|^{p} +| V^{\frac{1}{2}}f |^{p} \leq \alpha^{p}, \,\, \textrm{ a.e in}\,F.
 \end{equation}
 Let $(Q_k)$ be a Whitney decomposition of $\Omega$ by dyadic cubes so to say $\Omega$ is the disjoint union of the $Q_k$'s, 
the cubes $2Q_i$ are contained in  $\Omega$ and have the bounded overlap property, but the cubes $4Q_k$ intersect $F$.\footnote{In fact, the factor 2 should be some $c=c(n)>1$ explicitely given in [\cite{St},Chapter 6]. We use this convention to avoid too many irrelevant constants.}
Hence $$\sum_{k} | 2 Q_k | \leq C | \Omega | \leq C{\alpha}^{-p}\int_{\mathbb{R}^n} | L f|^{p} +| V^{\frac{1}{2}}f |^{p} .$$ Thus, \eqref{eq:czd'} and 
 \eqref{eq:cze'} are satisfied by $2 Q_k$.\\
 \\
 \textbf{b) Construction of $b_{k}$:}\\ 
 \\
 Let $(\chi_k)$ be a partition of unity on $\Omega$
associated to the covering $(Q_k)$ so that for each $k$, $\chi_k$ is a $C^1$ function supported in $2Q_k$ with 
 \begin{equation}\label{part}
 \| \chi_{k} \|_\infty + R_k \| \nabla {\chi_{k}} \| _{\infty} \leq c(n),
 \end{equation}
  where $R_k$ is the sidelength of $Q_k$ and $\sum \chi_{k} =1$ on $\Omega$.
  We say that a cube $Q$ is of type 1 if 
  $R^{2}\avert {Q} V > 1$,
 and is of type 2 if $R^{2}\avert {Q} V > 1$.
  
   We apply the gauge transformation on the cubes
$ 2Q_k$ such that $ Q_k$ is of type 2, hence there exist $\textbf{h}_{k}\in L^{n}(2Q_{k},\mathbb{R}^{n})$ and a real function $\phi_{k} \in H^{1}(2Q_{k})$ such that
\begin{equation}\textbf{h}_{k}=\textbf{a}- \nabla \phi \qquad \textrm{in}\,\, 2Q_{k},
 \end{equation}
and
 \begin{equation}\label{jaugecz'}\sup_{2 Q_{k}} | \textbf{h}|\leq  C\,R_{k}\,\sup_{2Q_{k}} |B|\leq C R_{k} \avert {2Q_{k}}V.
\end{equation}
We denote  $$m_{2Q_{k}}(e^{i\phi_{k}}f)= \avert{2Q_{k}} (e^{i\phi_{k}}f).$$ 
Let
 \begin{equation}
 b_{k}=\left\{\begin{array}{ll}
 f \chi_{k}, \,\textrm{if $2Q_{k}$ is of type 1,}\\
 \big{(}f-e^{-i\phi_{k}}m_{2Q_{k}}(e^{i\phi_{k}}f)\big{)} \chi_{k},\textrm{ if $2Q_{k}$ is of type 2.}

 \end{array}
 \right.
 \end{equation}
 \textbf{c) Proof of the estimate \eqref{eq:czc'}:}\\
 \\  
  Suppose $2Q_k$  is of type 1, $(2R_{k})^{2}\avert {2Q_{k}} V > 1$. Then $$(2R_{k})^{-p}\leq \big{(}\avert {2Q_{k}} V )^{p/2}\leq C \avert {2Q_{k}} V ^{p/2},$$
  here we used  $V^{p/2} \in RH_{2/p}$ since $p<2$.
  
Now we will control $L\,b_{k}$:
$$L b_{k}= L( f\chi_{k})=(Lf)\chi_{k}+\frac{1}{i}f\,\nabla\chi_{k},$$
then
$$\int_{2Q_k}| L b_{k} |^{p}+ R_{k}^{-p} | b_{k}|^{p}  \leq C \| \chi_{k} \|^{p} _{\infty}\int_{2Q_k}| L f |^{p} +\| \nabla\chi_{k} \|^{p} _{\infty}\int_{2Q_k} | f |^{p} +R_{k}^{-p} \| \chi_{k} \|^{p} _{\infty}\int_{2Q_k} | f |^{p}$$$$ \leq C \{\int_{2Q_k}| L f |^{p} + R_{k}^{-p}\int_{2Q_{k}}| f |^{p} \} \leq C \{ \int_{\tilde{Q_k}} | L f |^{p} + | V^{\frac{1}{2}}f |^{p}  \}\leq C\alpha^{p}  |{Q_k} |,$$
 here we used the $L^p$ version of the Fefferman-Phong inequality \eqref{eq:feff} and the intersection of $4Q_{k}$ with $F$. Hence estimation \eqref{eq:czc'} holds for the cubes of type 1.
 
  If $Q_k$ is of type 2,  $R_{k}^{2}\avert{Q_{k}} V \leq 1$. $V(x)dx$ is a doubling measure, then there exists $C>0$, such that $R_{k}^{2}\avert{Q_{2k}} |B| \leq C$.
 $$b_{k}=(f-e^{-i\phi_{k}}m_{2Q_{k}}(e^{i\phi_{k}}f)) \chi_{k}.$$
 Let us estimate $L\,b_{k}$. By the Gauge invariance, all we require is the estimation of $\tilde{L}(e^{i\phi_{k}}b_{k})$, where $$\tilde{L}=\frac{1}{i}\nabla-\textbf{h}_{k}.$$
We have
 $$ \tilde{L}(e^{i\phi_{k}}b_{k}) =  \chi_{k}(\tilde{L} f_{k})+ \frac{1}{i}(f_{k}-m_{2Q_{k}}f_{k}) \nabla\chi _{k}-\big{(}\avert{2Q_{k}} f_{k}\big{)}\, \chi_{k}\,  \textbf{h}_{k},$$
where $f_{k}=e^{i\phi_{k}}f$. Thus,
     $$\big{(}\avert{2Q_{k}} |L b_k|^{p}\big{)}^{1/p}\leq C \{ \big{(}\avert{2Q_{k}}|\tilde{L} f_{k}|^{p}\big{)}^{1/p}\, ||\chi_{k}||_{\infty}+ \big{(}\avert{2Q_{k}}|(f_{k}-m_{2Q_{k}}f_{k})|^{p}\big{)}^{1/p}\,|| \nabla\chi _{k}||_{\infty} $$$$+\big{(}\avert{2Q_{k}} |\textbf{h}_{k}|^{p}\avert{2Q_{k}}| f_{k}|^{p}\big{)}^{1/p}\, ||\chi_{k}||_{\infty}\}.$$
using Poincar\'e inequality and condition \eqref{part}, we obtain $$\big{(}\avert{2Q_{k}} |L b_k|^{p}\big{)}^{1/p}\leq C \{ \big{(}\avert{2Q_{k}}|\tilde{L} f_{k}|^{p}\big{)}^{1/p}+ \big{(}\avert{2Q_{k}}|\nabla f_{k}|^{p}\big{)}^{1/p} +\big{(}\avert{ 2Q_{k}} |\textbf{h}_{k}|^{p}\avert{2Q_{k}}| f_{k}|^{p}\big{)}^{1/p}\}$$

 $$\leq C \{ \big{(}\avert{2Q_{k}}|\tilde{L} f_{k}|^{p}\big{)}^{1/p}+ \big{(}\avert{2Q_{k}}|\frac{1}{i}\nabla f_{k}-\textbf{h}_{k}f_{k}|^{p}\big{)}^{1/p}$$$$+ \big{(}\avert{ 2Q_{k}} |\textbf{h}_{k}|^{p}\avert{2Q_{k}}| f_{k}|^{p}\big{)}^{1/p}+ \big{(}\avert{ 2Q_{k}}|\textbf{h}_{k}f_{k}|^{p}\big{)}^{1/p}\}$$
 $$\leq C \{ \big{(}\avert{2Q_{k}}|\tilde{L} f_{k}|^{p}\big{)}^{1/p}+C R_{k}\avert{ 2Q_{k}} V\,\big{(}\avert{2Q_{k}}| f_{k}|^{p}\big{)}^{1/p}\}$$
 $$\leq C \{ \big{(}\avert{2Q_{k}}|\tilde{L} f_{k}|^{p}\big{)}^{1/p}+ \big{(} \avert{ 2Q_{k}} V\big{)}^{\frac{1}{2}}\,\big{(}\avert{2Q_{k}}| f_{k}|^{p}\big{)}^{1/p}\}.$$
 Fefferman-Phong inequality \eqref{eq:FP} and $\big{(} \avert{ 2Q_{k}} V\big{)}^{\frac{1}{2}}\leq C \big{(} \avert{ 2Q_{k}} V^{p/2}\big{)}^{1/p}$ imply
 $$\big{(}\avert{2Q_{k}} |L b_k|^{p}\big{)}^{1/p}\leq C\big{(}\avert{2Q_{k}}|\tilde{L}f_{k}|^{p}+|V^{\frac{1}{2}}f_{k}|^{p}\big{)}^{1/p}.$$

Using gauge invariance, it follows $|L(f)|= |\tilde{L}(f_{k} )|$ and we deduce
$$\avert{2Q_{k}} |L b_k|^{p}\leq C \{ \avert{2Q_{k}}|L f|^{p}+|V^{\frac{1}{2}} f|^{p}\}\leq C \alpha^{p}.$$
Moreover,
 $$R_{k}^{-p}\avert{2Q_{k}} |b_k|^{p}=R_{k}^{-p}\avert{2Q_{k}} |(f_{k}-m_{2Q_{k}}f_{k}) \chi _{k}|^{p}\leq C \alpha^{p},$$
 here we used the previous argument.
Hence \eqref{eq:czc'} holds for $2Q_k$ of type $2$.\\
 \\
 \textbf{d) Definition and properties of $|B|^{\frac{1}{2}}g$:}\\ 
 \\
 Set $g=f- \sum b_{k}$. Note that, by \eqref{eq:cze'}, this sum is locally finite.
  It is clear that $g=f$ on $F$ and $g=\sum_{k\in J}e^{-i\phi_{k}}m_{2Q_{k}}(e^{i\phi_{k}}f) \chi_{k}$ on $\Omega$, where $J$ is the set of indices $k$ such that $Q_{k}$ is of type 2.
$$\int | V^{\frac{1}{2}}g|^{2}=\int_{F} | V^{\frac{1}{2}}g|^{2}+\int_{\Omega} | V^{\frac{1}{2}}g|^{2}=I+II.$$
By construction,
$$I=\int_{F} | V^{\frac{1}{2}}g|^{2}= \int_{F} | V^{\frac{1}{2}} f|^{2}\leq c \alpha^{2-p} \big{(}\| Lf \|_{p}+\| V^{\frac{1}{2}} f\|_{p}\big{)}^{p}.$$
  Using the $L^1$ version Fefferman-Phong inequality \eqref{eq:feff} we obtain
  $$II=\int_{\Omega} | V^{\frac{1}{2}}. g|^{2}\leq c \sum_{k\in  J}|Q_{k}|[\avert{2Q_{k}}V^{\frac{1}{2}}\avert{2Q_{k}}| f |]^{2}\leq C \sum_{k\in J}|Q_{k}|\alpha^{2}$$
  $$\qquad \leq c \alpha^{2}. \alpha^{-p} \int_{\mathbb{R}^n} |Lf |^{p}+| V^{\frac{1}{2}} f|^{p}.$$
Then
  \begin{equation}\label{bg'}
  \big{(}\int | V^{\frac{1}{2}}  g|^{2}\big{)}^{\frac{1}{2}} \leq c \alpha^{1-\frac{p}{2}} \big{(}\| Lf \|_{p}+\| V^{\frac{1}{2}} f\|_{p}\big{)}^{p/2}.
  \end{equation}
  \textbf{e)Calculation of $Lg$:}\\
\\
  Let $K$ the set of indices $k$. Let $\xi\in C^{\infty}_{0}(\mathbb{R}^{n})$, a test function. We know that, for all $k\in K$ such that $x\in 2Q_{k}$, there exists $C>0$ such that $d(x,F)> C\,R_{k}$. Therefore,
  $$\int \sum_{k \in K} |b_{k}| |\xi|\leq C \big{(} \int \sum_{k\in K} \frac{|b_{k}|}{R_{k}}\big{)} \sup_{x\in \mathbb{R}^n}\big{(} d(x,F) |\xi(x)|\big{)}.$$
  
The estimate \eqref{eq:czc'} gives us
$$\int  |b_{k}|^p\leq C{R_k}^p\alpha^{p} |Q_{k}|.$$
Hence
$$\int \sum_{k\in K} |b_{k}| |\xi|\leq C\alpha |\Omega| \sup_{x\in \Omega}\big{(} d(x,F) |\xi(x)|\big{)}.$$
We conclude that $\sum_{k\in K} b_{k}$ converges in the sense of distributions in $\mathbb{R}^n$. Then
$$\nabla g=\nabla f-\sum_{k\in K} \nabla b_{k},\,\, \textrm{in the sense of distributions in}\, \mathbb{R}^n.$$
 Since the sum is locally finite in $\Omega$ and vanishes on $F$, then $\textbf{a}\,g=\textbf{a}\,f - \sum_{k\in K}\textbf{a}\, b_{k}$ holds almost everywhere in $\mathbb{R}^n$. Hence 
$$Lg=Lf-\sum_{k\in K} L b_{k},\,\, \textrm{a.e in }\, \mathbb{R}^n.$$ 
\\
  \textbf{f) Proof of estimate \eqref{eq:czb'}:}\\
  \\ 
 To prove this inequality, we have to estimate $\|Lg\|_{2}$. It suffices to prove that $\|Lg\|_{\infty}\leq C \alpha$ since $\|Lg\|_{p}\leq C \big{(} \|Lf\|_{p}+ \|V^{\frac{1}{2}}f\|_{p})$).
We know that $\sum_{k\in K} \nabla \chi_{k}(x)=0$ for all $x\in \Omega$, then
 $$Lg = (L f) \mathbf{1}_{F}+ \sum_{k\in J}L(e^{-i\phi_{k}}\,m_{2Q_k}(e^{i\phi_{k}}f) \chi_{k})\qquad \textrm{a.e}. $$
 $$L( u)= e^{-i \phi_{k}} \tilde{L}(e^{i\phi_{k}} u) \quad \textrm{avec}\quad \tilde{L}=\frac{1}{i} \nabla -\textbf{h}_{k}.$$ Hence
 \begin{align*}\sum_{k\in J}L(e^{-i\phi_{k}}\,m_{2Q_k}(e^{i\phi_{k}}f) \chi_{k})&=\frac{1}{i}\sum_{k\in J}e^{-i\phi_{k}}m_{2Q_k}(e^{i\phi_{k}}f) \nabla\chi_{k}-\sum_{k\in J}e^{-i\phi_{k}}\,m_{2Q_k}(e^{i\phi_{k}}f) \chi_{k}\textbf{h}_{k}
\\
&=G_{1}+G_{2}.
\end{align*}
Now we will control $\|G_2\|_{\infty}$, we use \eqref{jaugecz'}, the $L^{1}$ version of Fefferman-Phong inequality \eqref{eq:feff}, the fact that $2Q_{k}$ is of type 2 and $V^{\frac{1}{2}}\in RH_{2}$ :
  $$|G_{2}(x)|=|\sum_{k\in J}m_{2Q_k}(e^{i\phi_{k}}f) \chi_{k}(x)\textbf{h}_{k}(x)|\leq C  \sum_{k\in J, x\in 2Q_{k}}R_{k}\avert{2Q_{k}}V|m_{2Q_k}(e^{i\phi_{k}}f)|$$$$ \leq C  \sum_{k\in J, x\in 2Q_{k}}\avert{2Q_{k}}V^{\frac{1}{2}}\avert{2Q_k}|f|\leq C N \alpha  .$$
Thus,
 \begin{equation}\label{eq:g2'}\|G_{2}\|_{\infty} \leq C \alpha.
 \end{equation}

Next, we estimate $\|G_1\|_{\infty}$. Remember that $G_{1}(x)=\sum_{k\in J}e^{-i\phi_{k}(x)}m_{2Q_{k}}(e^{i\phi_{k}}f) \nabla \chi_{k}(x)$.
 For all $m\in K$, consider $K_{m}=\{l\in K, 2Q_{l}\cap 2Q_{m}\neq\emptyset\}.$
 By construction of Whitney cubes, there exists a constant $c>0$ (we can take $c=18$) such that for any $m\in K$, $2Q_{l}\subset c \,Q_{m}$, for all $l\in K_m$. We denote $\tilde{Q}_{m}=cQ_m$. Let $\tilde{\phi}_{m}$ and $\tilde{h}_{m}$ the functions given by then gauge transform of Lemma \ref{th:jauge'} on $\tilde{Q}_m$. For a fixed $x$, we have
 
   \begin{align*}G_{1}(x)&=\sum_{k\in J}e^{-i\phi_{k}(x)}m_{2Q_{k}}(e^{i\phi_{k}}f) \nabla \chi_{k}(x)
\\
&=\sum_{k\in J } (e^{-i\phi_{k}(x)}m_{2Q_{k}}(e^{i\phi_{k}}f)-e^{-i\tilde{\phi}_{m}(x)}m_{2Q_{k}}(e^{i\tilde{\phi}_{m}}f)) \nabla \chi_{k}(x)
\\
&+\sum_{k\in J}e^{-i\tilde{\phi}_{m}(x)}\big{(}m_{2Q_{k}}(e^{i\tilde{\phi}_{m}}f)-m_{\tilde{Q}_{m}}(e^{i\tilde{\phi}_{m}}f)) \nabla \chi_{k}(x)
\\
&+\sum_{k\in J }e^{-i\tilde{\phi}_{m}(x)}m_{\tilde{Q}_{m}}(e^{i\tilde{\phi}_{m}}f) \nabla \chi_{k}(x)
\\
&=I+II+III.
\end{align*}
  First,
$$III=\sum_{k\in K_{m}} \chi_{m}(x)e^{-i\tilde{\phi}_{m}(x)}m_{\tilde{Q}_{m}}(e^{i\tilde{\phi}_{m}}f)  \nabla \chi_{k}(x)-\sum_{k\in K_{m}\setminus J} \chi_{m}(x)e^{-i\tilde{\phi}_{m}(x)}m_{\tilde{Q}_{m}}(e^{i\tilde{\phi}_{m}}f)  \nabla \chi_{k}(x).$$
The first term in $III$ vanishes since $\sum_{k\in K_{m}} \nabla \chi_{k}(x)=\sum_{k\in K} \nabla \chi_{k}(x)=0,$ for all $x\in 2Q_{m}$. Since all cubes $2Q_{k}$ with $k\in K_{m}\setminus J$ are of type 1, we obtain
$$\left|\sum_{k\in K_{m}\setminus J} \chi_{m}(x)e^{-i\tilde{\phi}_{m}(x)}m_{\tilde{Q}_{m}}(e^{i\tilde{\phi}_{m}}f)  \nabla \chi_{k}(x)\right|\leq C\sum_{k\in K_{m}\setminus J} R^{-1}_{k} \avert {\tilde{Q}_{m}}|f|$$
$$\leq C N\tilde{R}^{-1}_{m} \avert {\tilde{Q}_{m}}|f|\leq C \avert  {\tilde{Q}_{m}}|Lf|+ V|f|\leq C\alpha,$$
here we used $|Q_{k}|\sim |Q_{m}|$, \eqref{eq:cze'}, the Fefferman-Phong inequality and $4Q_{m}\cap F\neq \emptyset$.
Hence $\| III\|_{\infty} \leq C\alpha$.  

Secondly,
 $$|II|= |\sum_{k\in J, x\in 2Q_{k} }e^{-i\tilde{\phi_{m}}(x)}(m_{2Q_{k}}(e^{i\tilde{\phi}_{m}}f)-m_{\tilde{Q}_{m}}(e^{i\tilde{\phi}_{m}}f)) \nabla \chi_{k}(x) | $$$$\leq  \sum_{k\in J,x\in 2Q_{k} }|m_{2Q_{k}}(e^{i\tilde{\phi_{k}}}f)-m_{\tilde{Q}_{m}}(e^{i\tilde{\phi}_{m}}f)| ||\nabla \chi_{k}||_{\infty}$$$$ \leq C \sum_{k\in J,x\in 2Q_{k} } | m_{2Q_{k}}(e^{i\tilde{\phi}_{m}}f)-m_{\tilde{Q}_{m}}(e^{i\tilde{\phi}_{m}}f)| R_{k}^{-1},$$
since \begin{equation}\label{eq:av'}
| m_{2Q_{k}}(e^{i\tilde{\phi}_{m}}f)-m_{\tilde{Q}_{m}}(e^{i\tilde{\phi}_{m}}f)| \leq C \tilde{R_{m}}\alpha,
\end{equation}
then  
 $$\|II\|_{\infty}\leq C N\alpha.$$
 The proof of \eqref{eq:av'} is detailed in \cite{Aus}.

Finally, we will estimate $I$: 
\begin{align*}e^{-i\phi_{k}(x)}m_{2Q_{k}}(e^{i\phi_{k}}f)-e^{-i\tilde{\phi}_{m}(x)}m_{2Q_{k}}(e^{i\tilde{\phi}_{m}}f)
&=e^{-i\phi_{k}(x)}\avert{2Q_{k}}e^{i\phi_{k}(y)}f(y)\,dy-e^{-i\tilde{\phi}_{m}(x)}\avert{2Q_{k}}e^{i\tilde{\phi}_{m}(y)}f(y)\,dy
\\
&=\avert{2Q_{k}}\big{(}e^{i(\phi_{k}(y)-\phi_{k}(x))}-e^{i(\tilde{\phi}_{m}(y)-\tilde{\phi}_{m}(x))}\big{)}f(y)\,dy.
\end{align*}
Using inequality
 $$|e^{i(\phi_{k}(y)-\phi_{k}(x))}-e^{i(\tilde{\phi}_{m}(y)-\tilde{\phi}_{m}(x))}|\leq  |(\phi_{k}(y)-\phi_{k}(x))-(\tilde{\phi}_{m}(y)-\tilde{\phi}_{m}(x))|,$$
 we obtain$$|I| \leq \sum_{k\in J, x\in 2Q_{k}} R_{k}^{-1} \left| \avert {2Q_{k}} |f(y)||(\phi_{k}(y)-\phi_{k}(x))-(\tilde{\phi}_{m}(y)-\tilde{\phi}_{m}(x))|dy \right|.$$
 By construction, we have
  $$\nabla(\phi_{k}-\tilde{\phi_{m}})=\tilde{\textbf{h}}_{m}-\textbf{h}_{k}$$. We also have, for all $x$ and $y\in 2Q_{k}$
 $$\left|\big{(}\phi_{k}-\tilde{\phi}_{m}\big{)}(y)-\big{(}\phi_{k}-\tilde{\phi}_{m}\big{)}(x)\right|\leq \int^{1}_{0} |x-y| |\big{(}\tilde{\textbf{h}}_{m}-\textbf{h}_{k}\big{)}\big{(} x+t(y-x))|dt$$
 $$\leq C |x-y| R_{k} \avert {2Q_{k}} V\leq C R_{k}^{2} \avert {2Q_{k}} V \leq  C R_{k} \big{(} \avert{2Q_{k}} V^{\frac{1}{2}}\big{)},$$
 here we make use of \eqref{jaugecz'}, the fact that $2Q_{k}$ is of type 2 and $V^{\frac{1}{2}} \in RH_{2}$.

Hence $$|I| \leq \sum_{k\in J,x\in 2Q_{k}} \avert {2Q_{k}} V^{\frac{1}{2}} \avert {2Q_{k}} |f| \leq CN\alpha\, \, \textrm{a.e}.$$

It follows 
\begin{equation}\label{eq:g1'}\|G_{1}\|_{\infty} \leq C \alpha.
\end{equation}
We have
 $L g = (L f) \mathbf{1}_{F}+ G_{1}+G_{2}$ almost everywhere.

Since $| L f| \leq C\alpha$ on $F$, then using estimates \eqref{eq:g1'} and \eqref{eq:g2'}, we obtain
  \begin{equation}\|Lg\|_{\infty} \leq C \alpha.
 \end{equation}
Hence  $$\| L g\|_{2}+\| V^{\frac{1}{2}} g\|_{2} \leq C \alpha^{1-\frac{p}{2}}\big{(}\|Lf \|_{p}+\|V^{\frac{1}{2}}f\|_{p}\big{)}^{p/2},$$
 then estimate \eqref{eq:czb'} holds.
  \end{proof}
  \begin{rem}\label{th:1c'}Note that we did not use the fact that $V$ is the electrical potential of $H$: $V$ can be replaced by any  weight function $\omega$ in $A_{\infty}$:
 
 Suppose $\textbf{a}\in L^{2}_{loc}(\mathbb{R}^{n})^{n}$ and we assume the following condition for any cube  $Q$ in $\mathbb{R}^n$:
\begin{equation}\label{starstar'}  \sup_{Q} |B| \leq C \avert {Q} \omega.
\end{equation}
Then for any $1\leq p<\infty$, there exists a constant $C_{p}>0$, which depends on \eqref{starstar'}, such that
 \begin{equation}\| H(\textbf{a},0)^{\frac{1}{2}}(f)\|_{p}\leq C_{p}\| Lf\|_{p}+ \| \omega^{\frac{1}{2}}f\|_{p},
 \end{equation}
for all $f\in C_{0}^{\infty}(\mathbb{R}^{n}),\textrm{if}\,\, p>1,$ and
 \begin{equation}\label{eq:wt'}|\{x\in \mathbb{R}^n\, ; \,  |H(\textbf{a},0)^{\frac{1}{2}} f(x)| > \alpha\}| \le \frac{C_{1}}{\alpha} \int |L f| + \omega^{\frac{1}{2}}|f|,
\end{equation}
 for any $\alpha>0$ and $f\in C_{0}^{\infty}(\mathbb{R}^{n}),$ if $p=1$.
 \end{rem}

 \section{Estimates for weak solution}\label{sec:weaksol}
 Fix an open set $\Omega$. A subharmonic function on $\Omega$ is a function $v\in L^{1}_{loc}(\Omega)$ such that $\Delta v \ge 0$ in $D'(\Omega)$.

 \begin{lemm}\label{subh}Suppose $\textbf{a}\in  L^{2}_{loc}(\mathbb{R}^{n})^{n}$ and $0 \leq V
\in  L^{1}_{loc}(\mathbb{R}^{n})$. If $u$ is a weak solution of
$H(\textbf{a},V)u=0$ in $\Omega$, then $|u|^{2}$ is a subharmonic function and 
\begin{equation}\Delta  |u|^{2}= 2 |Lu|^{2} + 2 V |u|^{2}.
\end{equation}
\end{lemm} 
\begin{proof}Since
$$\Delta |u|^{2} =\Delta (u \overline{u})=2 Re((\Delta u) \overline{u}) + 2 |\nabla u|^{2}, $$
and
$H(\textbf{a},V)u=0$, then
$$\Delta u = \sum^{n}_{k=1}( i a_{k}\frac{\partial u}{\partial x_{k}}  + i\frac{\partial}{\partial x_{k}}(a_{k}u))+ |\textbf{a}|^{2}u+ Vu.$$
It follows that
\begin{align*}\Delta |u|^{2}&=2 Re\bigg{(}\sum^{n}_{k=1}( i a_{k}\frac{\partial u}{\partial x_{k}}+ i\frac{\partial}{\partial x_{k}}(a_{k}u))\, \overline{u}+ |\textbf{a}|^{2}u\overline{u}+ Vu \overline{u}\bigg{)}+ 2 |\nabla u|^{2} 
\\
&=2 Re\bigg{(}\sum^{n}_{k=1}( i a_{k}\frac{\partial u}{\partial x_{k}} \,\overline{u}+i \frac{\partial}{\partial x_{k}}(a_{k}u)\,\overline{u}\bigg{)}+ 2|\textbf{a}|^{2}|u|^{2}+ 2V|u|^{2}+ 2 |\nabla u|^{2} 
\\
&=2 Re\bigg{(}\sum^{n}_{k=1}( i a_{k}\frac{\partial u}{\partial x_{k}}\,\overline{u}+ i\frac{\partial}{\partial x_{k}}(a_{k}|u|^{2})-i a_{k} u\frac{\partial \overline{u}}{\partial x_{k}}\bigg{)}+ 2|\textbf{a}|^{2}|u|^{2}+ 2V|u|^{2}+ 2 |\nabla u|^{2}
\\
&=4 Im(\textbf{a} \nabla u \overline{u})+ 2|\textbf{a}|^{2}|u|^{2}+ 2 |\nabla u|^{2}+ 2|Vu|^{2}= 2 |Lu|^{2} + 2V|u|^{2} .
\end{align*}

\end{proof}
 \begin{cor}
  Let $Q$ a cube in $\mathbb{R}^n$ and $u$ a weak solution of $H(\textbf{a},V) u =0$ in a neighbourhood of $\overline{2Q}$. For
  $V\ge 0$, we have the following inequality
  \begin{equation}
\label{eq:mvi'}
\sup_{Q} |u| \le C(r,n,\mu) \big(\avert {\mu Q} |u|^r\big)^{1/r},
\end{equation}
for any $0<r<\infty$ and $1<\mu\le 2$. 
\end{cor}

The following technical lemma is interesting in its own right. For a detailed proof see \cite{Buc} and \cite{AB}. It states that a form of the  mean value inequality for subharmonic functions still holds if the Lebesgue measure is replaced 
by a weighted measure of Muckenhoupt type. More precisely, 
\begin{lemm}\label{th:sousharm} 

 	Let  $\omega\in RH_{q}$ for some $1<q\le \infty$ and let  $0<s<\infty$ and $r>q$ (if $q=\infty$, $r=\infty$) such that $\omega\in RH_{r}$.  Then there exists a constant  $C\ge 0$  depending only on $\omega$,$r$,$p$,$s$ and $n$,  such that  for any  cube $Q$  and any nonnegative  subharmonic function  $f$ in  a neighbourhood of $\overline{2Q}$  we have  for all $1< \mu \le 2$, 
$$
\big(\avert{Q} (\omega f^s)^{r} \big)^{1/r} \le C \, \avert{\mu Q} \omega f^s,\,\textrm{for}\, \,r<+\infty.
$$
And
$$
\sup_{Q}f \le \frac{ C}{\avert{Q}\omega} \, \avert{\mu Q} \omega f^s,\,\textrm{for}\,\, r=+\infty.
$$
\end{lemm}


\textbf{Throughout this section we will assume  $V\in RH_{q}$ with $1<q\leq +\infty$ and $B$ satisfies the assumption \eqref{eq:star} and $u$ is a weak solution of $H(\textbf{a},V)u=0$ in $4Q$}. We will establish some local estimates on $|u|$ and $|Lu|$. Using the gauge transform on $4Q$, we can replace 
$\textbf{a}$ and $L$  by $\textbf{h}$ and $\frac{1}{i} \nabla -\textbf{h}$ as defined previously in Lemma \ref{th:jauge'}. 

All the constants are independant of $Q$ and $u$ but they may depend on $V$ and $q$. 
 
  First we give three important results that are the main tools for the proof of Theorem \ref{th:1a'}
 
  \begin{pro}\label{th:csv'}There exists a constant $C>0$ such that
\begin{equation} \label{csv}\big{(} \avert{Q}| V^{\frac{1}{2}} u |^{2q} \big{)}^{\frac{1}{2}q}\leq C \big{(} \avert{3Q}| V^{\frac{1}{2}} u |^{2} \big{)}^{\frac{1}{2}}.
\end{equation}

  \end{pro}
  \begin{proof}It is a consequence of Lemma \ref{th:sousharm} and Lemma \ref{subh}.
  \end{proof}
  \begin{pro}
  \label{lemm:5'}  
For all $k>0$, there exists a constant $C$ such that
\begin{equation}\label{wrhL}
\big( \avert { Q}   |L u|^{ q^*} \big)^{1/ q^*} \le   \frac {C} { (1+ R^2\avert {Q}V)^k} \, \big(\avert {3 Q}  |L u|^2 + V|u|^2 \big)^{\frac{1}{2}}.
\end{equation}
\end{pro}

\begin{pro}
\label{lemm:6'}  
Let $1< \mu\le 4$, if  $n/2 \le q <n$ then there exists a constant $C$ such that
\begin{equation}\label{wrhL2}
\big( \avert { Q}   |L u|^{ q^*} \big)^{1/ q^*} \le  C \, \big(\avert {\mu Q}   |L u|^{2} \big)^{\frac{1}{2}}, 
\end{equation}
If $q \ge n$ then there exists a constant $C$ such that
\begin{equation}\label{wrhL2'}
\sup_{ Q}   |L u|  \le  C \, \big(\avert {\mu Q}   |L u|^{2} \big)^{\frac{1}{2}}. 
\end{equation}
\end{pro}
\begin{rem}\label{rem6}Using Theorem 2 of \cite{IN}, we can replace $2$ by $\delta \in ]0,2]$ in \eqref{wrhL2} and \eqref{wrhL2'}.
\end{rem}

We need the following results to prove propositions \ref{lemm:5'} and \ref{lemm:6'}:
  \begin{lemm}
  \label{lemm:1'}
Let $1 \le \mu < \mu' \le 4$ and $k>0$, then there exists a constant $C$ such that
$$ 
\avert {\mu Q}|u|^2 \le \frac {C} { (1+ R^2\avert {Q}V)^k} \big(\avert {\mu' Q} |u|^2\big),
$$
and
$$ 
\avert {\mu Q} |L u|^2 + V|u|^2 \le \frac {C} { (1+ R^2\avert {Q}V)^k} \big(\avert {\mu' Q} |L u|^2 + V|u|^2 \big).
$$
\end{lemm}

The proof is analogous to that of Lemma 8.1 in \cite{AB}, it is based on Caccioppoli \eqref{eq:cc} type inequality and the improved Fefferman-Phong inequality \eqref{eq:FP}.

\begin{lemm}
\label{lemm:2'}
For any $1<\mu \le 4$ and $k>0$, there exists a constant $C$ such that 
$$
(R\avert {Q}V)^2 \, \avert {Q} |u|^2 \le \frac {C} { (1+ R^2\avert {Q}V)^k}  \big(\avert {\mu Q}   V|u|^2 \big).
$$ 
\end{lemm}
\begin{proof} Using Lemma \ref{lemm:1'} with $k>1$ and $1<\mu' <\mu$ and subsequently Lemma \ref{th:sousharm}, we have:
$$
(R \avert{Q} V)^2 \avert{ Q}|u|^2 \le \frac {C  \avert{Q}V\,\avert{ \mu' Q}|u|^2} { (1+ R^2\avert{Q}V)^{k-1}}  \le \frac {C  \avert{\mu' Q}V\,\sup_{ \mu' Q}|u|^2} { (1+ R^2\avert{Q}V)^{k-1}}   \le \frac {C  \avert{\mu Q} (V|u|^2)} { (1+ R^2\avert{Q}V)^{k-1}}  .
$$
 
 \end{proof}
\begin{lemm}
\label{lemm:3'}
For any $1<\mu\le 2$, $k>0$  and $n <p<\infty$, there exists a constant $C$ such that
$$
(R\avert {Q}V)^2 \, \avert {Q} |u|^2 \le  \frac {C} { (1+ R^2\avert {Q}V)^k}  \big( \avert {\mu Q}   |L u|^p \big)^{2/p}.
$$
\end{lemm}
\begin{proof} If $\avert{\mu Q} |L u|^p
=\infty$ , there is nothing to prove. Assume, therefore, that $\avert{\mu Q} |L u|^p
<\infty$.
  Let $1<\nu <\mu$ and $\eta$ be a smooth non-negative function, bounded by 1, equal to 1 on $\nu Q$ with support on $\mu Q$ and whose gradient is bounded by $ C/R$ and Laplacian by $C/R^2$.

Integrating the equation $H(\textbf{a},0) u + Vu =0$ against� $\bar{u}\eta^2$. \\
Since

 $$H(\textbf{a},V)u = \sum_{j=1}^{n} L_{j}^{\star} L_{j} u + V u,$$
  
$$\int H(\textbf{a},V)u\,\bar{u}\eta^2 =  \sum_{j=1}^{n}\int  L_{j} u\,\overline{L_{j}(u\eta^2)} + \int V |u|^{2}\,\eta^2,$$
then
$$
\int  |L u|^2 \eta^2 + V|u|^2 \eta^2= 2 \int L u \cdot \nabla \eta\, \bar{u} \eta, $$
hence
$$\int  V|u|^2 \eta^2\le \frac C R \bigg( \int_{\mu Q} |L u|^2 \bigg)^{1/2} \bigg( \int |u|^2 \eta^2\bigg)^{1/2},$$
\begin{equation}\label{X}
X \le  {C\, {(R^2\avert{Q}V)^{1/2}}  |\mu Q|^{1/2}\, Y^{1/2}\, Z^{1/2}}
\end{equation}
where we set $X=(R^2\avert{Q}V) \int V |u|^2 \eta^2$,  $Y= \big(\avert{\mu Q} |L u|^p\big)^{2/p}$
and $Z= \avert{Q}V \int |u|^2 \eta^2$.
By Morrey's embedding theorem and diamagnetic inequality \eqref{2diam}, $u$ is H\"older continuous with exponent $\alpha= 1- n/p$. Hence for all $x,y\in \mu Q$, we have

$$
\left||u(x)| -|u(y)|\right| \le C \bigg( \frac{|x-y|}R\bigg)^\alpha \, R\, \big(\avert{\mu Q} |\nabla |u||^p\big)^{1/p}\leq C \bigg( \frac{|x-y|}R\bigg)^\alpha \, R\, Y^{1/2}.
$$
 We pick $y\in \overline Q$ such that $|u(y)|= \inf_{Q}|u|$. Then 
\begin{align*}
Z=\avert{Q}V \int |u|^2 \eta^2 &\le 2  (\avert{Q} V) \inf_{Q}|u|^2 \int \eta^2 + 2 (\avert{Q} V) \int \left||u(x)| -|u(y)|\right|^2 \eta^2(x)\, dx
\\
& \le 2 \big(\avert{Q} (V|u|^2)\big) \int \eta^2 + C (\avert{Q} V) R^2 Y \, \int 
 \bigg( \frac{|x-y|}R\bigg)^{2\alpha} \eta^2(x)\, dx
 \\
 & \le C   \big(\avert{Q} (V|u|^2)\big)  |Q| + C (\avert{Q} V) R^2 Y\,|\mu Q|  \\
 & \le C  \int V|u|^2\eta^2 + C (\avert{Q} V) R^2 Y\,|\mu Q|.
 \end{align*}
 where, in the penultimate inequality, we used the support condition on $\eta$ and $0\le \eta \le 1$, and in the last,    $\eta=1$ on $ Q$. 
 Using the previous inequalities, we obtain
 $$X  \le C  |\mu Q| ^{1/2} \, Y^{1/2} \, \big( CX + C(R^2\avert{Q}V)^2  |\mu Q| Y\big)^{1/2},
 $$
 which, as $2ab \le \epsilon^{-1} a^2 + \epsilon b^2$ for all $a,b \ge 0$ and $\epsilon>0$, implies
  $$
 X \le C(1 + R^2 \avert{Q}V )^2 \,  |\mu Q| \, Y.
 $$
 Next, let $1< \nu'<\nu$. Using $\eta=1$ on $\nu Q$ Lemma \ref{th:sousharm} and Lemma \ref{lemm:1'}
 $$
 \int  V|u|^2 \eta^2 \ge \int_{\nu Q} V |u|^2 \ge C \avert{\nu' Q } V \, \int_{\nu' Q} |u|^2 \ge C   (\avert{Q } V)  (1 + R^2 \avert{Q}V )^k \int_{ Q} |u|^2,
 $$
 hence
 $$
 X \ge C (R \avert{Q}V)^2  (1 + R^2 \avert{Q}V )^k \int_{Q} |u|^2.
 $$
 The upper and lower bounds for $X$ yield the lemma.

 \end{proof}
Now we will give the proof of \ref{lemm:5'} and \ref{lemm:6'}.\\
\\
\paragraph{\textit{Proof of Proposition \ref{lemm:5'}:}}
We will assume that $q>\frac{2n}{n+2}$.\\

\textbf{a) Preparation:}\\
\\
We remind that  modulo a gauge transformation,  $u$ is a weak solution of $H(\textbf{h},V)u=0$ on $4Q$, where $\textbf{h}$ is the potential function defined in Lemma \ref{th:jauge'}. We call $L=\frac{1}{i} \nabla -\textbf{h}$.%

Let $v$ a weak solution of $\Delta v=0$ in $ 2 Q$ with $v=u$ on $\partial( 2 Q)$ . Set $w=u-v$ in $ 2 Q$. From elliptic theory we know that
 $$\avert { 2Q}   |\nabla w|^{2} \leq  \avert {2 Q}   |\nabla u|^{ 2} ,$$
 thus,
  $$\avert { 2Q}   |\nabla v|^{2} \leq 4 \avert {2 Q}   |\nabla u|^{ 2} .$$
  It suffices to establish the following inequality for $1\leq \mu <2$,
  \begin{equation}\label{13}\sup_{\mu Q} |Lv|\le C \big( \avert {\frac{3}{2} Q}   |L u|^{ 2} + V|u|^{2} \big)^{\frac{1}{2} },
  \end{equation}
 and
 \begin{equation}\label{2} \big( \avert { Q}   |Lw|^{ q^*} \big)^{1/ q^*}\le \phi\big( R^{2} \avert {Q} V\big)\big( \avert {\frac{5}{2} Q}   |L u|^{ 2} + V|u|^{2} \big)^{\frac{1}{2} },
 \end{equation}
where $\phi$ is a real function which increases polynomially.
The conclusion follows easily through Lemma \ref{lemm:1'}. \\
 \\
 \textbf{b) Estimate of $v$:}\\
 \\
 Since $\nabla v$ is harmonic on $2Q$ then for $1\leq \mu <2$, 
$$\sup_{\mu Q} |\nabla v|\le C  \avert {2 Q}   |\nabla v|^{ 2} \leq 4C \avert {2Q} |\nabla u|^{2}.$$
Using the harmonicity of $v$ and the fact that $v=u$ on $\partial (2Q)$, we obtain
$$ \sup_{2Q} |v|^{2} \leq \sup_{\partial(2Q)} |v|= \sup_{\partial(2Q)} |u|,$$
and \eqref{eq:mvi'} gives
\begin{equation}\label{3}\sup_{2Q}|v| \leq C\big( \avert{\frac{9}{4}Q}|u|^{2}\big)^{\frac{1}{2}}.
\end{equation}
 \eqref{eq:star}, \eqref{jaugebes'}, Lemma \ref{lemm:2'} and the fact that $\avert {2Q} V \sim \avert {\frac{9}{4}Q} V,$ yield
\begin{equation}\label{4}\sup_{2Q} |\textbf{h} v| \leq C\big[\big(R\avert {2Q} V\big)^{2} \avert {\frac{9}{4}Q}|u|^{2}\big]^{\frac{1}{2}}\leq C\big[ \avert {\frac{5}{2}Q}V|u|^{2}\big]^{\frac{1}{2}}.
\end{equation}
Hence, inequality \eqref{13} holds.\\ 
\\
\textbf{c) Estimate of $w$:}\\
\\
Using $w=u-v$, \eqref{3} and \eqref{eq:mvi'}, we obtain
\begin{equation}\label{5} \sup_{2Q} |w| \leq C \big( \avert {\frac{9}{4}Q} |u|^{2}\big)^{\frac{1}{2}},
\end{equation}
and
\begin{equation}\label{6}\sup_{2Q} |\textbf{h}w| \leq C \big( \avert {\frac{5}{2}Q} V|u|^{2}\big)^{\frac{1}{2}}.
\end{equation}
It suffices to estimate $\nabla w$. We have
$$\avert {2Q} |\nabla w|^{2}\leq \avert {2Q} |\nabla u|^{2}\leq 4 \avert {2Q} |Lu|^{2}+ \avert {2Q} |\textbf{h}u|^{2}, $$
by combining \eqref{4} and \eqref{6}, it follows
\begin{equation}\label{7}\avert {2Q} |\nabla w|^{2}\leq C \avert {\frac{5}{2}Q} |L u|^{2}+ V|u|^{2}.
\end{equation}
   Now we will control $\avert {Q} |\nabla w|^{q^*}$. Let $1<\mu <\mu' <2$ and $\eta$ be a smooth non-negative function, bounded by 1, equal to 1 on $\mu Q$ with support on $\mu' Q$ and whose gradient is bounded by $ C/R$ and Laplacian by $C/R^2$. We know that 
   $$0= H(\textbf{h},V)(u)= -\Delta u-\frac{1}{i} \nabla. (\textbf{h}u)- \frac{1}{i}\textbf{h}. \nabla u+ ( |\textbf{h}|^{2}+V)u,\, \, \textrm{on}\,\, 4Q.$$
   Then
   $$\Delta u=-\frac{1}{i}( \nabla. \textbf{h})u- \frac{2}{i}\textbf{h}. \nabla u+ ( |\textbf{h}|^{2}+V)u$$ 
    $$=-\frac{1}{i}( \nabla. \textbf{h})u- 2\textbf{h}. L u+ ( -|\textbf{h}|^{2}+V)u.$$ 
    Since $\Delta w=\Delta u$ on $2Q$, then
    $$ \Delta (w\eta)=(\Delta w)\eta +2 \nabla w. \nabla \eta+ w \Delta \eta$$
    $$=-\frac{1}{i}( \nabla. \textbf{h})u\eta- 2(\textbf{h}. L u)\eta+ ( -|\textbf{h}|^{2}+V)u\eta+ 2 \nabla w.\nabla \eta+ w\Delta \eta.$$
Let $\Gamma_{0}$ be the fundamental solution of $\Delta$. We know that
$$|\nabla \Gamma_{0}(x)|\leq C |x|^{1-n}.$$
 Hence, for $x\in \mu' Q$,
\begin{align*}|\nabla w(x)|&\leq C \big( \int_{\mathbb{R}^{n}} \frac{|\textbf{h}(y) Lu(y)| \eta(y)}{|x-y|^{n-1}}dy+\int_{\mathbb{R}^{n}} \frac{\big(|\nabla\textbf{h}(y)|+ | \textbf{h}(y)|^{2}\big)|u(y)| \eta(y)}{|x-y|^{n-1}}dy
\\
&+\int_{\mathbb{R}^{n}} \frac{V(y) |u(y)| \eta(y)}{|x-y|^{n-1}}dy+ \int_{\mu' Q \setminus \mu Q }\big( \frac{|\nabla w(y)|}{R|x-y|^{n-1}}+\frac{|w(y)|}{R^{2}|x-y|^{n-1}}\big) dy\big).
\end{align*}
Using inequalities \eqref{jaugebes'}, \eqref{jaugebes''} and \eqref{eq:star}, we obtain
$$|\nabla w(x)|\leq C \big( \big( R\avert {Q}V\big) \int_{\mathbb{R}^{n}} \frac{| Lu(y)| \eta(y)}{|x-y|^{n-1}}dy+\big( \avert {Q} V +R \big(\avert {Q}V\big)^{\frac{3}{2}}+  \big(R\avert {Q}V\big)^{2}\big)\int_{\mathbb{R}^{n}} \frac{|u(y)| \eta(y)}{|x-y|^{n-1}}dy$$
\begin{equation}\label{8}+\int_{\mathbb{R}^{n}} \frac{V(y) |u(y)| \eta(y)}{|x-y|^{n-1}}dy+ \{\avert{2 Q }|\nabla w|+ \frac{1}{R}|w|\}=I+II+III+IV.
\end{equation}
 Poincar\'e inequality and \eqref{7} imply
$$\avert {2Q} |\nabla w| + \frac{1}{R}|w|\leq C\big( \avert {2Q} |\nabla w|^{2}\big)^{\frac{1}{2}}\leq C \big( \avert {\frac{5}{2}Q} |Lu|^{2}+ V |u|^{2}\big)^{\frac{1}{2}}. $$
Now we have to estimate $III$. We use the Hardy-Littlewood-Sobolev theorem, the fact that $V\in RH_{q}$, estimate \eqref{eq:mvi'} and Lemma \ref{lemm:2'}:
$$\big(\avert {\mu' Q}III^{q^*}\big)^{\frac{1}{q^*}}\leq C R \big( \avert {\mu Q}|Vu|^{q}\big)^{\frac{1}{q}}\leq C R\avert {2Q} \sup_{2Q} |u|$$
$$\leq C \big( \avert{\frac{5}{2}Q}|Lu|^{2}+ V |u|^{2}\big)^{\frac{1}{2}}.$$
Since
$$\int_{\mathbb{R}^{n}} \frac{|u(y)| \eta(y)}{|x-y|^{n-1}}dy\leq \sup_{2Q}|u|\int_{2Q}\frac{dy}{|x-y|^{n-1}}\leq CR \sup_{2Q}|u|$$
Then, \eqref{eq:mvi'} and Lemma \ref{lemm:2'} imply :
$$\sup_{\mu'Q} |II|\leq C\big(1+ \big(R\avert {Q}V\big)^{\frac{1}{2}}+R\avert {Q} V\big) R\avert {Q} V \sup_{2Q}|u|$$
$$\leq C \big( 1+ R\avert {Q} V\big) \big( \avert {\frac{5}{2}Q} |Lu|^{2}+ V|u|^{2}\big)^{\frac{1}{2}}.$$
Finally, we apply Hardy-Littlewood-Sobolev theorem and we obtain
$$\big( \avert {\mu' Q}I^{q^*}\big)^{\frac{1}{q^*}}\leq C R^{2}\avert {Q} V \big( \avert {\mu Q} |Lu|^{q}\big)^{\frac{1}{q}}.$$
Hence
$$\big( \avert {\mu' Q} |\nabla w|^{q^*}\big)^{{1}{q^*}}\leq C \big( 1+ R^{2}\avert {Q} V\big)\big( \avert {\frac{5}{2}Q}|Lu|^{2}+ V|u|^{2}\big)^{\frac{1}{2}}+ C R^{2}\avert {Q} V \big( \avert {\mu Q} |Lu|^{q}\big)^{\frac{1}{q}}.$$
Since $Lw=Lu-Lv$, we combine the above inequality with \eqref{13} :
$$\big( \avert {\mu' Q}|Lw|^{q^*}\big)^{\frac{1}{q^*}}\leq C \big( 1+ R^{2}\avert {Q}V\big)\big(\avert {\frac{5}{2}Q}|Lu|^{2}+ V|u|^{2}\big)^{\frac{1}{2}}+ C\big(R^{2} \avert {Q}V\big) \big( \avert {\mu Q}|Lw|^{q}\big)^{\frac{1}{q}}.$$
By iterating the previous inequality and using \eqref{7} we finish the proof of \eqref{2}.$\square$\\

\textit{Proof of Proposition \ref{lemm:6'}:} The proof is the same as that of Proposition \ref{lemm:5'}, we use  Lemma \ref{lemm:2'} instead of Lemma \ref{lemm:3'}.

Let us quote two additional results which ensue at once from what precedes but which we shall not use. 
\begin{cor}
\label{lemm:4'}
Let $1< \mu\le 2$ and $k>0$. If $q<n$, then there exists a constant $C$ such that
\begin{equation}\label{pt}
\big( \avert { Q}   |L u|^{ q^*} \big)^{1/ q^*} \le  \frac {C} { R(1+ R^2\avert {Q}V)^k} \big( \sup_{\mu Q} |u|\big).
\end{equation}
If  $q \ge n$, then there exists a constant $C$ such that
\begin{equation}\label{gd}
\sup_{ Q}   |L u|  \le    \frac {C} { R(1+ R^2\avert {Q}V)^k}  \big(\sup_{\mu Q} |u|\big). 
\end{equation}

\end{cor}
\begin{proof}It remains to use Caccioppoli type inequality \eqref{eq:cc} and Proposition \ref{lemm:6'}.
\end{proof}

\begin{cor}
\label{lemm:7'} 
Let $ q\ge n/2$. For any $1< \mu\le 2$ and $k>0$ there exists a constant $C$ such that
$$
(R  \avert { Q} V)^2  \,  \avert { Q} | u|^2 \le   \frac {C} { (1+ R^2\avert {Q}V)^k}  \big(\avert {\mu Q} |L u|^2\big). 
$$
\end{cor}
\begin{proof} The proof is a consequence of Lemma \ref{lemm:4'} and Proposition \ref{lemm:6'}.
\end{proof}

\section{Maximal inequalities}

We give some important maximal inequalities that we will use to prove our results around the Riesz transforms.  These estimates are essentially a consequence of some $L^p$ estimates proved in \cite{AB} and based on the $L^1$ inequality deduced from the work of Gallou\"et and Morel \cite{GM} in the semi-linear setting or from Kato's paper \cite{K2}.
 \begin{theo}\label{th:VH}  Let $\textbf{a}\in  L^{2}_{loc}(\mathbb{R}^{n})^{n}$ and $ V
\in RH_{q},\quad 1< q\leq +\infty$. Then, there exists $ \epsilon >0$, depending only on the $RH_q$ constant of  $V$, such that $V\,H(\textbf{a},V)^{-1}$ and $H(\textbf{a},0)H(\textbf{a},V)^{-1}$ are $L^p$ bounded for all $1 \leq p
< q+\epsilon$.
\end{theo}

\begin{cor}\label{th:VH2}  Let $\textbf{a}\in  L^{2}_{loc}(\mathbb{R}^{n})^{n}$ and $ V
\in RH_{q}, \, \, 1<q\leq +\infty$. Then, there exists an $ \epsilon >0$, depending only on the $RH_q$ constant of  $V$, such that, the operators $V^{1/2}\,H(\textbf{a},V)^{-1/2}$ and $H(\textbf{a},0)^{1/2}\, H(\textbf{a},V)^{-1/2}$ are $L^p$ bounded for all $1< p< 2q+\epsilon$.
\end{cor}
 To prove this result, we shall use complex interpolation  relying on the fact that for all $y\in \mathbb{R}$, the imaginary power of Schr\"odinger operator  $H^{iy}$  has a bounded extension on $\mathbb{R}^n$, $1<p<\infty$. This result due to Hebisch \cite{H} follows from the Gaussian estimates on the heat kernel $e^{-tH}$ proved by \cite{DR} .
Here, $H^{iy}$ is defined  as  a bounded operator on $L^2(\mathbb{R}^n)$ by functional calculus ( see \cite{AB} for more details).\\

 \underline{\textit{Proof of Theorem \ref{th:VH}:}}\\
 	
The proof of this theorem is identical to that of Theorem 1.1 in \cite{AB}. First we prove an $ L ^ 1$ inequality, then we establish some reverse H\"older type estimates, then finally we apply Theorem \ref{theor:shen}.

\begin{lemm}\label{lemm:l1}Let $f \in L^{\infty}_{comp}(\mathbb{R}^n)$ and $u=H(\textbf{a},V)^{-1}f$. Then, 
\begin{equation}\label{2Vu}
\int_{\mathbb{R}^n} V|u| \le \int_{\mathbb{R}^n} |f|, 
\end{equation}
and
\begin{equation}\label{2Hu}
\int_{\mathbb{R}^n} |H(\textbf{a},0)u| \le 2 \int_{\mathbb{R}^n} |f|.
\end{equation}
 
\end{lemm}

\begin{proof}
 $V\geq 0$, by Kato-Simon inequality \eqref{eq:KS}, we have
 $$|H(\textbf{a},V)^{-1}f|\leq  H(0,V)^{-1}|f|.$$
 We know, by \cite{AB} that
 $$\int_{\mathbb{R}^n} VH(0,V)^{-1}|f| \le \int_{\mathbb{R}^n} |f|.$$ Thus, inequality
 \eqref{2Vu} holds, and inequality \eqref{2Hu} follows by difference.
\end{proof}

\underline{\textit{ Proof of the $L^p$ maximal inequality:}} \\

Assume $V\in RH_q$ with $q> 1$.	
$V H (\textbf{a}, V) ^{-1} $. We know that this operator is bounded on $ L ^ 1 (\mathbb{R}^ n) $, so we apply Theorem \ref{theor:shen} through the reverse H\"older inequality verified by any weak solution.
Set $Q$ a fixed cube and $f\in L^\infty(\mathbb{R}^n)$ a function with compact support in $\mathbb{R}^n \setminus 4Q$. Then $u=H(\textbf{a},V)^{-1} f$ is well defined in  $ \dot{\mathcal{V}}$ and it is a weak solution of $H(\textbf{a},0) u + Vu=0$ in $4Q$.

Since $|u|^2$ is subharmonic, by Lemma \ref{th:sousharm} with $w=V$, $f=|u|^2$ and $s=1/2$, we obtain
$$\big{(}\avert{Q} |Vu|^{q}\big{)}^{1/q} \leq C \avert{2Q} |Vu|.$$
Thus \eqref{T:shen} holds with $T=V H(\textbf{a},V)^{-1}$,  $p_{0}=1$, 
$q_{0}=q$, $S=0$, $\alpha_{1}=2$ and $\alpha_{2}=4$. Hence $V H(\textbf{a},V)^{-1}$ is bounded on $L^p(\mathbb{R}^n)$ for all $1<p<q$ by Theorem \ref{theor:shen}. Due to the properties of $RH_{q}$ weights, we can replace $q$ by $q+\epsilon$. Taking the difference, we obtain the same result for $H(\textbf{a},0) H(\textbf{a},V)^{-1}$.
This completes the proof of Theorem \ref{th:VH} .$\square$

\section{Proof of Theorem \ref{th:1a'}} 

Now we will focus on the proof of Theorem \ref{th:1a'}  we first need to establish some important inequalities. First of all we  reduce the problem thanks to the following Lemma:

 \begin{lemm}Under the assumptions of Theorem \ref{th:1a'}. For any $p>2$, the $L^p$ boundedness of $L H(\textbf{a},V)^{-\frac{1}{2}}$ is equivalent to that of  $L H(\textbf{a},V)^{-1} L^{\star}$ and $L H(\textbf{a},V)^{-1}V^{\frac{1}{2}}$.
 
 \end{lemm}
 \begin{proof}
 If $LH(\textbf{a},V)^{-\frac{1}{2}}$ is $L^p$ bounded, then by duality and using $L^p$ boundedness of Riesz transforms for $1<p\leq 2$ , we have $H(\textbf{a},V)^{-\frac{1}{2}}L^{\star}$ is $L^{q}$ bounded for any $q\geq 2$. Hence $L H(\textbf{a},V)^{-1} L^{\star}$ is $L^p$ bounded.
 By the same argument, $H(\textbf{a},V)^{-\frac{1}{2}} V^{\frac{1}{2}}$ is $L^p$ bounded, and hence $LH(\textbf{a},V)^{-1}V^{\frac{1}{2}}$ is $L^p$ bounded.
 
Reciprocally, if $LH(\textbf{a},V)^{-1}L^{\star}$ and $LH(\textbf{a},V)^{-1}V^{\frac{1}{2}}$ are $L^p$ bounded, then their adjoints $LH(\textbf{a},V)^{-1}L^{\star}$ and $V^{\frac{1}{2}}H(\textbf{a},V)^{-1}L^{\star}$ are $L^{p'}$ bounded.

 Let $\textbf{F}\in C ^{\infty}_{0}(\mathbb{R}^{n}, \mathbb{C}^{n})$, then $$\|H(\textbf{a},V)^{-\frac{1}{2}}L^{\star}\textbf{F}\|_{p'}=\|H(\textbf{a},V)^{\frac{1}{2}}H(\textbf{a},V)^{-1}L^{\star}\textbf{F}\|_{p'}.$$
 Using assumption \eqref{eq:star} and inequality \eqref{eq:revLp}, we obtain
$$\|H(\textbf{a},V)^{-\frac{1}{2}}L^{\star}\textbf{F}\|_{p'}\leq C\big(\|LH(\textbf{a},V)^{-1}L^{\star}\textbf{F}\|_{p'}+\|V^{\frac{1}{2}}H(\textbf{a},V)^{-1}L^{\star}\textbf{F}\|_{p'}\big)\leq C \|F\|_{p'}.$$
Thus $LH(\textbf{a},V)^{-\frac{1}{2}}$ is $L^p$ bounded.
 \end{proof}
 Next we look at some useful related estimates  
\begin{pro} \label{thC'}Let $V \in RH_q$, $1<q\leq +\infty$. Then, for any $2<p< 2(q+\epsilon)$, and $\epsilon>0$ depending only on $V$,  $f\in C^\infty_{0}(\mathbb{R}^n,\mathbb{C})$ and $\textbf{F}\in C_{0}^\infty(\mathbb{R}^n, \mathbb{C}^n)$,
$$
\|V^{\frac{1}{2}} H(\textbf{a},V)^{-1} V^{\frac{1}{2}}f\|_p  \le C_{p} \|f\|_{p}, \quad  \|V^{\frac{1}{2}} H(\textbf{a},V)^{-1} L^{\star} \textbf{F}\|_p \le C_p\|\textbf{F}\|_p.
$$
\end{pro}
\begin{proof}
T effectuate the proof, we apply Theorem \ref{theor:shen} and eventually use equation
\eqref{csv}.
\end{proof}
To prove Theorem \ref{th:1a'}, it suffices to prove the following Proposition:
\begin{pro} \label{thD'}
Let $V \in RH_q$, $1<q\leq+\infty$. Then, for any $2<p<  q^*+\epsilon$ and $\epsilon>0$ depending only on $V$, $f\in C^\infty_{0}(\mathbb{R}^n,\mathbb{C})$ and $\textbf{F}\in C_{0}^\infty(\mathbb{R}^n,\mathbb{C}^n)$,
$$
\|L  H(\textbf{a},V)^{-1} V^{\frac{1}{2}}f\|_p  \le C_{p} \|f\|_{p}, \quad  \|L H(\textbf{a},V)^{-1} L^{\star} \textbf{F}\|_p \le C_p\|\textbf{F}\|_p.
$$
\end{pro}
 \begin{proof}
Suppose $q<n/2$. 
 Let $Q$ a cube in $\mathbb{R}^n$ and $\textbf{F}\in C^{\infty}_{0}(\mathbb{R}^n)$ supported away from $4Q$. Set $H=H(\textbf{a},V)$, $u=H^{-1} L^{\star}\textbf{F}$ is well defined on $\mathbb{R}^n$. In particular, the support condition on $\textbf{F}$, implies that $u$ is a weak solution of $Hu=0$ in $4Q$. Hence, using Proposition \ref{lemm:5'}, we have
 \begin{equation}
  \big{(} \avert{Q}| LH^{-1}L^{\star} \textbf{F}|^{q^*}dx \big{)}^{1/q^*}\leq C \big{(} \avert{3Q}| L H^{-1} L^{\star}\textbf{F}|^{2}+|V^{\frac{1}{2}}H^{-1}L^{\star}\textbf{F}|^{2} \big{)}^{\frac{1}{2}}.
  \end{equation}
  Then \eqref{T:shen} holds with $T=LH^{-1}L^{\star},\, q_{0}=q^{\star},\, p_{0}=2 \, \textrm{and}\, S=M_{2}(V^{\frac{1}{2}}H^{-1}L^{\star}),$ where $M_{2}f=\big(M|f|^{2}\big)^{\frac{1}{2}}$ and $M$ is the maximal Hardy-Littlewood operator. Since $S$ is $L^p$ bounded for any $1<p\leq 2q$, then, using Proposition \ref{thC'} and the fact that $q^{\star}\leq 2q$, $T$ is bounded on $L^p(\mathbb{R}^{n},\mathbb{C}^{n})$, for $2<p<q^{\star}$.
 
 Now, let $ n/2 \le q < n$. By the same method and using Proposition \ref{lemm:6'}, we obtain 
 \begin{equation}
  \big{(} \avert{Q}| LH^{-1}L^{\star} \textbf{F}|^{q^*}dx \big{)}^{1/q^*}\leq C \big{(} \avert{3Q}| L H^{-1} L^{\star}\textbf{F}|^{2}\big{)}^{\frac{1}{2}}.
  \end{equation}
 Hence, inequality \eqref{T:shen} holds with $T=LH^{-1}L^{\star},\, q_{0}=q^{\star}\, \textrm{and}\,  p_{0}=2$. Thus, $T$ is bounded on $L^p$, $2<p< q^{\star}$. Finally, if $q\ge n$, then $L H^{-1} L^{\star}$ is $L^p$ bounded for $2<p < \infty$.
 
  Using the same argument with $u=H^{-1}V^{\frac{1}{2}}f$, it follows that $LH^{-1}V^{\frac{1}{2}}$ is $L^p$ bounded for $2<p<q^{*}+\epsilon$.

\end{proof}

\section{Second order Riesz transforms}
This section is devoted to the study of some operators using results previously established. We are interested in the behaviour of $L_{j} L_{k} H(\textbf{a},V)^{-1}$ and $V^{\frac{1}{2}} L H(\textbf{a},V)^{-1}$. We also need some properties of the kernel of $H(\textbf{a},V)^{-1}$. Note that, for $V\in L^1_{loc}(\mathbb{R}^n)$, $V\ge 0$ and $\textbf{a}\in L^{2}_{loc}(\mathbb{R}^n)^n$, $H(\textbf{a},V)^{-1}$ is continuous from $L^{1}(\mathbb{R}^n)$ to $L^1_{loc}(\mathbb{R}^n)$.Hence there exists  $\Gamma(x,y)$ a Schwartz kernel associated to this operator .
The following proposition gives some results about this kernel
\begin{pro}\label{pro:fundsol} $\Gamma$ coincides with a measurable function defined on $\mathbb{R}^n\times \mathbb{R}^n$ in  $ \mathbb{C}$, and
\begin{enumerate}
\item We have the following inequality
\begin{equation}\label{*}|\Gamma (x,y)| \leq C_{n} |x-y|^{2-n} \, \, \textrm{a.e}.
\end{equation}
\item For almost every $y\in \mathbb{R}^{n}$, $u:x\longmapsto \Gamma (x,y)$ is a weak solution of $H(\textbf{a},V)(u)=0$ on $\mathbb{R}^{n}\setminus \{y\}$. 
\end{enumerate}

\end{pro}
\begin{proof}Let $H=H(\textbf{a},V)$. Inequality \eqref{dom} and $H^{-1} =\int^{\infty}_{0} e^{-tH} dt$ imply that $\Gamma$ is dominated by the Green function of the Laplacien. Hence, inequality \eqref{*} follows.
Now we will prove (2). Let $f\in C^{\infty}_ {0}(\mathbb{R}^{n})$, using (1) we have the following integral representation 
$$H^{-1} (f) (x)= \int_{\mathbb{R}^{n}} \Gamma (x,y) f(y) dy$$
for almost every $x$.

Fix a function $\rho \in C^{\infty}_{0}(\mathbb{R}^{n})$, even,  supported in $[-1,1]^{n}$ with $\int \rho=1$.\\ Set $\rho^{y}_{j}=2^{nj} \rho(2^{j}(z-y))$ for any $y,\, z \in \mathbb{R}^{n}$ and $j\in \mathbb{N}$. Using Fubini theorem, we obtain 
$$H^{-1}(\rho^{0}_{j}*f)(x)= \int_{\mathbb{R}^{n}}\Gamma_{j}(x,y)f(y) dy,\,\, a.e$$
where
$$\Gamma_{j} (x,y)= \int_{\mathbb{R}^{n}} \Gamma (x,z) 2^{nj}\rho(2^{j}(z-y))dz=H^{-1}(\rho^{y}_{j})(x).$$
Using Lebesgue's dominated convergence theorem, we obtain:
$$\lim_{j\rightarrow \infty}H^{-1}(\rho^{0}_{j}*f)=H^{-1}(f)\,\, \textrm{a.e},$$
and $$\lim_{j\rightarrow \infty} \Gamma_{j}=\Gamma \, \, \textrm{a.e in }\, L^{1}_{loc}.$$
Since $\rho^{y}_{j}\in C^{\infty}_{0}(\mathbb{R}^{n})$, $H^{-1}(\rho^{y}_{j})\in \dot{\mathcal{V}}$ and $H(\Gamma_{j}(., y))=\rho^{y}_{j} \in \mathcal{D}'(\mathbb{R}^{n})$. Hence $\Gamma_{j}(.,y)$ is a weak solution of $Hu=0$ away from the support of $\rho^{y}_{j}$, i.e $\mathbb{R}^{n}\setminus Q(y,2^{j})$. Here $Q(c,R)$ is the cube centred in c with sidelength $R$.
By \eqref{*} and a similar argument to that used for the proof of Caccioppoli type inequality, we obtain for $R>2^{-j+2}$,
$$\int_{\mathbb{R}^{n}\setminus Q(y, 2R)} |L_{x}\Gamma_{j}(x,y)|^{2}+ V(x) |\Gamma_{j}(x,y)|^{2} dx\leq \frac{C}{R^{2}}\int_{Q(y,2R)\setminus Q(y, R)}|\Gamma_{j}(x,y)|^{2}\leq CR^{2-n}.$$
 Hence, for any $y$, $L_{x}\Gamma_{j}(.,y)$ and $V^{\frac{1}{2}}\Gamma_{j}(., y)$ admit subsequences that weakly converge in $L^{2}_{loc}(\mathbb{R}^{n}\setminus\{y\})$. It is easy to prove that their limits are $L_{x}\Gamma(., y)$ and $V^{\frac{1}{2}}\Gamma(., y)$ in $\mathcal{D}'(\mathbb{R}^{n})$ for almost every $y\in \mathbb{R}^{n}$. We deduce that for almost every $y\in \mathbb{R}^{n}$, $\Gamma(., y)$ is a weak solution of $Hu=0$ on $\mathbb{R}^{n}\setminus\{y\}$ (by taking limits on the equation).

\end{proof}
Proposition \ref{pro:fundsol} and the following technical lemma are two important tools to prove the main results of this section.
\begin{lemm}\label{lem*}Let $V\in RH_{n/2}$. Then
$$\sum_{l\in\mathbb{Z}}\frac{\big(4^{l}\avert {Q(y, 2^{l})}V\big)^{\frac{1}{2}}}{1+ 4^{l} \avert{Q(y,2^{l})}V}\leq C(V)$$
uniformly on $y\in \mathbb{R}^{n}$.

\end{lemm}
\begin{proof}
Shen proved in \cite{Sh1}, Lemma 1.2, that there exists an $\alpha>0$ such that for any $r,\, R$ with $r<R$ and $y\in \mathbb{R}^{n}$,
$$r^{2} \avert {Q(y,r)}V \leq C\big( \frac{r}{R}\big)^{\alpha} R^{2} \avert {Q(y,R)}V.$$
Let $j$ be the biggest integer such that $4^{j} \avert{Q(y, 2^{j})}V<1$. We  have
$$\sum_{l\in\mathbb{Z}}\frac{\big(4^{l}\avert {Q(y, 2^{l})}V\big)^{\frac{1}{2}}}{1+ 4^{l} \avert{Q(y,2^{l})}V}\leq C(\alpha )\big(\big(4^{j}\avert {Q(y, 2^{j})}V\big)^{\frac{1}{2}}+\big( 4^{j+1} \avert{Q(y,2^{j+1})}V\big)^{-\frac{1}{2}}\big)\leq2 C(\alpha).$$

\end{proof}
\begin{rem}This lemma does not hold for $V(x)=\frac{1}{|x|^{2}} \in \cup_{q<\frac{n}{2}}RH_{q}$.
\end{rem}

Now we will study the $L^p$ boundedness of $V^{\frac{1}{2}}LH(\textbf{a},V)^{-1}$:

\begin{pro}\label{pr}Suppose $\textbf{a}\in L^{2}_{loc}(\mathbb{R}^{n})^{n}$ and $V  \in RH_{n/2}$. Also assume that there exists a constant $C>0$ such that for any cube $Q$ in $\mathbb{R}^{n}$:

 \begin{equation}
 \left\{\begin{array}{ll}
 \sup_{Q} |B| \leq C \avert {Q} V, \\ \sup_{Q} | \nabla B |\leq C  (\avert{Q} V)^{3/2}.
  \end{array}
 \right.
 \end{equation}
 Then, $V^{\frac{1}{2}} L  H(\textbf{a},V)^{-1}$ is $L^1$ bounded.
\end{pro}
\begin{proof}It suffices to prove 
$$esssup_{y\in \mathbb{R}^{n}}\int_{\mathbb{R}^{n}}V^{\frac{1}{2}}(x)|L_{x}\Gamma(x,y)|dx <+\infty.$$
To simplify the proof we will take $y=0$. Set $\Gamma(.,0)$ the weak solution of $Hu=0$ on $\mathbb{R}^{n}\setminus \{0\}$. Let $(Q_{l,k})$ be a Whitney decomposition of $\mathbb{R}^{n}\setminus \{0\}$:  $d(Q_{l,k},0)\sim 2^{l}$ the sidelength of $Q_{l,k}$ and $k$ belongs the finite set of indices of cardinality $2^{n}-1$. Hence, using the Cauchy-Schwarz inequality and \eqref{*}.
\begin{align*}\int_{\mathbb{R}^{n}} V^{\frac{1}{2}}(x)|L_{x}\Gamma(x,0)| dx &=\sum_{l,k}\int_{Q_{l,k}} V^{\frac{1}{2}}(x)|L_{x}\Gamma(x,0)|dx
 \\
& \leq C\sum_{l,k}|Q_{l,k}| \big( \avert {Q_{l,k}} V\big)^{\frac{1}{2}} \big( \avert {Q_{l,k}} |L_{x}\Gamma (x,0)|^{2}dx\big)^{\frac{1}{2}}
\\
& \leq C\sum_{l,k}|Q_{l,k}| \big (\avert {Q_{l,k}} V\big)^{\frac{1}{2}}\big(1+4^{l}\avert {Q_{l,k}} V\big)^{-1} \big(4^{-l} \avert {Q_{l,k}} |\Gamma (x,0)|^{2}dx\big)^{\frac{1}{2}}
\\
& \leq \sum_{l,k} \frac{\big (4^{l}\avert {Q_{l,k}} V\big)^{\frac{1}{2}}}{(1+4^{l}\avert {Q_{l,k}} V\big)}.
\end{align*}
Since $V$ is in $A_{\infty}$, then $\avert {Q_{l,k}} V\sim \avert {Q(0,2^{l})}V$ uniformly on $l$ and $k$. Hence by Lemma \ref{lem*} as $V\in RH_{n/2}$.
\end{proof}

\textbf{Proof of Theorem \ref{vl}:} Since $T=V^{\frac{1}{2}} L H(\textbf{a},V)^{-1}$ is $L^1$ bounded , then it suffices to apply Theorem \ref{theor:shen} to the operator $T$  with $p_{0}=1$ and $q_{0}=\frac{2qn}{3n-2q}$ if $q<n$ and $q_{0}=2q$ if $q\geq n$. 

Let $Q$ a cube of $\mathbb{R}^n$ and $f\in L^\infty_{comp}(\mathbb{R}^n)$ with support away from $4Q$.
We know that $u=H(\textbf{a},V)^{-1}f$ is a weak solution of $H(\textbf{a},V)u =0 $ in the neighbourhood of $4Q$.
 There exists $s<q^*$ such that
$\frac 1 {q_{0}} = \frac 1 {2q} +  \frac 1 s$.
Then, 
$$\big(  \avert{Q} (V^{\frac{1}{2}} |L u | )^{q_{0}} \big)^{1/q_{0}}  \le \big(  \avert{Q} V^{q} \big)^{\frac{q}{2}} \big(  \avert{Q}  |L u | ^{s} \big)^{1/s}.$$
Using the remark \ref{rem6} and $V^{\frac{1}{2}} \in RH_{2r}$ ( we used 
Proposition \ref{11.1}) we obtain
$$\big(  \avert{Q} (V^{\frac{1}{2}} |L u | )^{q_{0}} \big)^{1/q_{0}}\leq C \big(  \avert{Q} V^{\frac{1}{2}} \big) \big(  \avert{2 Q}  |L u | ^{\delta } \big)^{1/\delta },$$
for any $\delta >0$.

 Let $\delta >0$ such that $V^{\frac{1}{2}} \in A_{1/\delta }$, the Muckenhoupt class. We know that (\cite{St},chap V)
$$\big(  \avert{2 Q}  |g | ^{\delta } \big)^{1/\delta }\leq C\big( \frac 1 {  \avert{2Q}V^{\frac{1}{2}}}\big) \int_{2Q} V^{\frac{1}{2}}  |g |,\, \,\, \textrm{for any measurable and non-negative function $g$ }.$$
Hence
$$\big(  \avert{Q} (V^{\frac{1}{2}} |L u | )^{q_{0}} \big)^{1/q_{0}}\leq C  \avert {2Q} (V^{\frac{1}{2}}  |L u |),$$
and \eqref{T:shen} holds with $S=0$. Thus $T$ is $L^p$ bounded for any $1<p<q_{0}$.

We finish the proof using the self-improvement of the reverse H\"older classes.

Now we will use the previous results to get the proof of Theorem \ref{max}.

\paragraph{\textit{Proof of Theorem \ref{max}:}}

Set $H(\textbf{a},V)=H$ and $H(\textbf{a},0)=H_{0}$.  $$L_{s}L_{k} H^{-1}= L_{s} H^{-1} L_{k} + L_{s} [L_{k},H^{-1}].$$
  For every $j\geq 1$, $L_{j}H^{-\frac{1}{2}}$ is $L^p$ bounded for any $1<p<\infty$ then $L_{s} H^{-1} L_{k}$ is $L^p$ bounded for any $1<p< \infty$.
 Since
 $$[L_{k},H^{-1}]= -H^{-1} [L_{k},H]H^{-1},$$
$$[L_{k},H]=[L_{k},H(\textbf{a},0)]+ [L_{k},V].$$
Then $$L_{s}H^{-1} [L_{k},V]H^{-1}=(L_{s}H^{-1}L_{k})(V H^{-1})-(L_{s}H^{-1}V^{\frac{1}{2}})( V^{\frac{1}{2}}L_{k}H^{-1}).$$
Proposition \ref{thD'}, Theorem \ref{th:VH} and Theorem \ref{vl} imply the $L^p$ boundedness of $L_{s}H^{-1} [L_{k},V]H^{-1}$ for any $1<p<q$.

We would now like to study the behaviour of $L_{s}H^{-1} [L_{k},V]H^{-1}$.
We know that $$[L_{k},H_{0}]=\sum_{j}b_{kj}L_{j}- \partial_{j}b_{kj} . $$ 
Here $b_{kj}$ are $\partial_{j} b_{kj}$ the operators of multiplication by $b_{kj}$ and $\partial_{j}b_{kj}$.
 
It follows through \eqref{eq:star}
 $|B(x)|\leq C V(x)$ and $|\nabla B(x)|\leq CV^{3/2}(x)$, for almost every $x\in \mathbb{R}^{n}$. Hence, $L_{s} H^{-1} b_{kj} L_{j} H^{-1}$ is $L^p$ bounded if $L_{s} H^{-1}V L_{j} H^{-1}$ is $L^p$ bounded. Moreover $$L_{s} H^{-1}V L_{j} H^{-1}=(L_{s} H^{-1}V^{\frac{1}{2}})\,(V^{\frac{1}{2}} L_{j} H^{-1}),$$
 is $L^p$ bounded for $1<p< q_{0}$, where $q_{0}=\frac{2q+n}{3n-2q}$ if $q<n$ and $2q$ if $q\ge n$. Here we have used Proposition \ref{thD'} and Theorem \ref{vl}.\\ 
 $L_{s} H^{-1} \partial_{j}b_{kj}  H^{-1}$ is $L^p$ bounded if $L_{s} H^{-1}V^{3/2} H^{-1}$ is $L^p$ bounded. We could also say
 $$L_{s} H^{-1}V^{3/2} H^{-1}=(L_{s} H^{-1}V^{\frac{1}{2}})\,( VH^{-1}),$$
which is bounded for $1<p<q$ using Proposition \ref{thD'} and Theorem \ref{th:VH}.
  Hence, $L_{s} [L_{k},H^{-1}]$ is $L^{p}$ bounded for $1<p< q+\epsilon$. $\square$

\end{document}